\documentclass[11pt,
]{amsart}

\usepackage{
amssymb, graphicx
}

\numberwithin{equation}{section}%

\newtheorem{theorem}{Theorem}[section]
\newtheorem{proposition}{Proposition}[section]

\newtheorem{definition}{Definition}[section]
\newtheorem{corollary}{Corollary}[section]

\theoremstyle{definition}
\newtheorem{remark}{Remark}[section]

\DeclareMathOperator{\supp}{supp}
\DeclareMathOperator{\Ker}{Ker}

\DeclareMathOperator{\spec}{spec}

\DeclareMathOperator{\Ran}{Ran}
\newcommand{\eps}{\varepsilon}

\newcommand{\R}{{\bf R}}
\newcommand{\Id}{\mbox{Id}}
\renewcommand{\r}[1]{(\ref{#1})}
\newcommand{\PDO}{$\Psi$DO}
\newcommand{\be}[1]{\begin{equation}\label{#1}}
\newcommand{\ee}{\end{equation}}

\renewcommand{\d}{\mathrm{d}}

\newcommand{\bo}{\partial \Omega}
\renewcommand{\L}{\mathcal{L}}

\title[Multiwave tomography with reflectors]{Multiwave tomography with reflectors: Landweber's iteration}

\author[P. Stefanov]{Plamen Stefanov}
\address{Department of Mathematics, Purdue University, West Lafayette, IN 47907}
\thanks{First author partly supported by  NSF  Grant DMS-1301646}
\date{\today}
\author[Yang Yang]{Yang Yang}
\address{Department of Mathematics, Purdue University, West Lafayette, IN 47907}

\begin{document}
\begin{abstract}
We use the Landweber method for numerical simulations for the multiwave tomography problem with a reflecting boundary and compare it with the averaged time reversal method. We also analyze the rate of convergence and the dependence on the step size for the Landweber iterations on a Hilbert space. 
\end{abstract}
\maketitle

\section{Introduction} 
In this paper, we apply the well known   Landweber iteration method for the multiwave tomography  problem with reflectors. We also compare it to the Averaged Time Reversal (ATR) method developed by the authors in \cite{St-Y-AveragedTR} leading to an inversion of a Neumann series.  We also present a convergence analysis for the Landweber method for linear problems in Hilbert spaces. Even though the Landweber iteration method is widely used and well studied, 
we prove some results that we have not found in the literature. 
In numerical inversions, we always work in finite dimensional spaces, but there are good reasons to understand the method better in the continuous model first, and then to understand how well the discretization approximates the continuous model. 
As we demonstrate in this work, for the inverse problem we study involving wave propagation, the discrete model based on a second order finite difference scheme is always unstable regardless of whether the continuous one is stable or not. On the other hand, stable problems in our case behave in a stable way numerically, even though they are not (when discretized). We analyze the convergence rates of the iterations, their relation to the a priori stability of the problem,  of the spectrum of the operator and on the spectral measure of the object we want to recover, and present numerical evidence of those results. 

While there are more advanced computational inversion methods for general linear or non-linear inverse problems, including such with regularization, our goal is not to develop ``the best one''.  Without real life data and understanding the real measurements challenges, that would be not so useful anyway. We would like to have methods that would allow us to test the mathematics of the problem: how the geometry affects the stability, for example; and partial data on the boundary. We are also interested in how the discretization and the numerical solver to model the PDE affects the problem, regardless of how it is solved after that. 

The multiwave tomography  problem with reflectors (the term ``cavity'' is often used, as well) we study models  the following. We send one type of waves to an object we want to image, like a part of the human body like microwave radiation (in thermoacoustic tomography), laser rays (in photoacoustic tomography) or elastic waves. Those waves have a low resolution (the laser ray would diffuse inside) but they are partly absorbed by the cells, which  cause them to emit ultrasound waves of high enough frequency allowing for a high resolution. The emitted sound waves are measured on the boundary, or on a part of it, and we want to recover the source; after that, we want to recover the absorption rate. We do not study the latter problem here. 

This problems is well  studied under the assumption that the waves travel freely away from the body, see, e.g., \cite{finchPR,Kruger03,Kruger99, KuchmentK_11,XuWang06, SU-thermo, SU-thermo_brain, QSUZ_skull, S-U-InsideOut10, ScherzerL, ScherzerL2,Arridge_2016} and the references there. In some experimental setup, see  \cite{CoxAB_07}, reflectors are placed around the body and the measurements are made on a part of it. The mathematical model of this is presented below in \r{BVP}, first studied in  \cite{KunyanskiHC_2014,H_Kunyanski_14, ScherzerL, ScherzerL2}. In \cite{St-Y-AveragedTR} and \cite{KunyanskyNguyen2015}, see also \cite{Acosta_M}, two different stable ways to solve the problem were proposed, both based on exponentially convergent Neumann series, expanding some ideas introduced first in \cite{SU-thermo}. The inversion in \cite{Acosta_M, KunyanskyNguyen2015} works for the case of partial boundary data as well; while  in this case, in \cite{St-Y-AveragedTR}, we know that it gives a parametrix for partial boundary data but there is no proof yet that in converges to a solution. We tested it in that case anyway, and we also do it in this paper. Recently,  Acosta and Montalto  \cite{AcostaM_16} studied a model taking account attenuation, see also \cite{Andrew13}.

Landweber iterations for the whole space problem were done recently in \cite{ScherzerL, ScherzerL2} and compared to the sharp time reversal method in \cite{SU-thermo} implemented numerically in \cite{QSUZ_skull}. Another recent work is \cite{Arridge_2016}, where several different optimizations are compared, based on different regularization terms. 

\section{The model  and known theoretical results} 
Let $\Omega\subset\mathbb{R}^n$ be a bounded open subset. 
 Consider 
\begin{equation}  \label{BVP}
\left\{\begin{array}{rcll}
(\partial^2_t - c^2(x) \Delta )u = & 0 & \quad\quad \text{ in } (0,T)\times\Omega,  \\
\partial_\nu u |_{(0,T)\times\partial\Omega} = & 0,  & \\
u|_{t=0} = & f, & \\
\partial_t u|_{t=0} = & 0 ,&
\end{array} \right.
\end{equation}
where $\partial_\nu$ is the normal derivative. One could consider a Riemannian metric $g$ and then $\Delta$ would be the Laplace-Beltrami operator, see  \cite{St-Y-AveragedTR} and section~\ref{sec_ATR}. 
The measurement operator is
\be{L}
\mathcal{L}:f\mapsto u|_{(0,T)\times\partial\Omega}.
\ee
The goal is to inverse $\mathcal{L}$ and recover $f$. In the partial data problem, we are given 
\be{LG}
\mathcal{L}_\Gamma:f\mapsto u|_{(0,T)\times\Gamma},
\ee
where $\Gamma\subset \bo$ is a relatively open subset.

We recall some uniqueness and stability results about this problem, see also \cite{St-Y-AveragedTR}. Uniqueness follows from unique continuation \cite{tataru95}. The sharpest time $T$ for uniqueness if given by 
\be{Pr1}
T_0 := \max_{x\in\bar\Omega_0}\text{dist}(x,\Gamma),
\ee
where $\text{dist}(x,\Gamma)$ is the distance between a point and a set in the metric $c^{-2}\d x^2$. If $T<T_0$, $\L f$ determines $f$ uniquely only at the points staying at distance less than $T$ from $\Gamma$. If $T>T_0$, there is a global uniqueness.

A sharp stability condition follows from  the  Bardos, Lebeau and Rauch \cite{BardosLR_control}. In view of the intended applications, we would assume that $\bo$ is strictly convex and that $\supp f$ is compactly supported in $\Omega$ but \cite{BardosLR_control} covers the general case.

\begin{definition} \label{def1}
Let $\bo$ be strictly convex with respect to $c^{-2}\d x^2$. 
Fix  $\Omega_0\Subset \Omega$, an open  $\Gamma\subset\Omega$ and $T>0$.

 (a)  We say that the stability condition is satisfied if every broken unit speed geodesic $\gamma(t)$  with $\gamma(0)\in \bar\Omega_0$ has at least one common  point  with $\Gamma $ for $|t|<T$, i.e., if    $\gamma(t)\in \Gamma$ for some $t\in (-T,T)$. 
 
 (b) We call the point $(x,\xi)\in T^*\bar\Omega_0\setminus 0$ a {visible singularity} if the unit speed geodesic $\gamma$ through $(x,\xi/|\xi|)$ has a common point with $\Gamma$ for $|t|<T$. We call the ones for which $\gamma$ never reaches $\bar\Gamma$ for $|t|\le T$ invisible ones. 
\end{definition} 

We denote by $T_1=T_1(\Gamma,\Omega)$ the least upper bound of all $T$ for which the stability condition is satisfied. 

As mentioned in \cite{St-Y-AveragedTR}, the visible and invisible singularities do not cover the while $T^*\Omega_0\setminus 0$ since some rays may hit $\partial\Gamma$ or $t=T$. The complement of those two sets is of measure zero however. 

Visible singularities can be recovered stably, and non-visible cannot. We will not give formal definitions (see, e.g., \cite{SU-JFA09}). In numerical computations, this means that visible singularities, like edges, etc., can in principle be recovered well (the would look ``sharp'') but  the actual reconstruction may require some non-trivial efforts. Non-visible singularities cannot be recovered well and the typical way to deal with this is to recover some regularized version of them, like blurred edges. 

If the stability condition is satisfied, them one has an $H^1\to H^1$ stability  estimate \cite{St-Y-AveragedTR}. The proof of this follows from the fact that $\L$ is an FIO of order $0$ associated with a local diffeomorphism. That diffeomorphism can be described in the following way. For every $(x,\xi)\in T^*\Omega_0\setminus 0$, we take the geodesic $\gamma_{x,\xi/|\xi|}(t)$, where $|\xi|$ is the norm of the covector $\xi$ at $x$ in the metric $c^{-2}\d x^2$. We also identify vectors and covectors by the metric. When $\gamma_{x,\xi/|\xi|}$ hits $\bo$, we take that point and the projection of the tangent vector on $\bo$. That correspondence is a local diffeomorphism, and here it is essential that we have guaranteed that $\gamma_{x,\xi/|\xi|}$ hits $\bo$ transversely. By the stability condition, such a contact with $\bo$ exists. After that, we reflect the geodesic by the usual law of reflection, and look for another contact, etc. We consider both positive and negative $t$. That would give us a multi-valued map but locally (near the starting singularity and locally near the image), each branch is a diffeomorphism. Then $\L$ is a restriction of a first order elliptic FIO (say, the same operator but defined on a larger open set containing the closure of  $(0,T)\times\Gamma$), to $(0,T)\times\Gamma$. In particular, that means that it is bounded from $L^2(\Omega_0)$ to $L^2((0,T)\times\Gamma)$. We consider $L^2(\Omega_0)$ as subspace of $L^2(\Omega)$. 
The stability condition guarantees that at least one of the contacts with $\bo$ is interior for that set, and this is enough for building a microlocal parametrix. This, together with the uniqueness result, implies the following theorem.  We refer to \cite{SU-JFA09} for a more general case, and to \cite{BardosLR_control} for this particular one. 

\begin{theorem}\label{thm_st}
Let $\bo$ be strictly convex and fix  $\Omega_0\Subset \Omega$, an open $\Gamma\subset\Omega$ and $T>0$. Then if the stability condition is satisfied, i.e., if $T>T_1$, then there exist $0<\mu\le C$ so that
\[
\mu \|f\|_{L^2(\Omega_0)}\le \|\L f\|_{L^2((0,T)\times\Gamma)}\le C\|f\|_{L^2(\Omega_0)}, \quad\forall f\in L^2(\Omega_0). 
\]
\end{theorem}

In fact, in the reconstructions, we work in the topologically equivalent space $L^2(\Omega_0,c^{-2}\d x)$. This is the measure appearing in the energy  $\int_\Omega(|\nabla f_1|^2+|f_2|c^{-2})\,\d x$; and our numerical simulations show that the error is smaller if we use that measure in the definition of $\L^*$ later.

If there are invisible singularities, then the estimate above cannot hold  with $\mu>0$ and it fails even if we replace the $L^2$ norm by weaker Sobolev ones. 

We review the Averaged Time Reversal method in section~\ref{sec_ATR}. One of the essential differences is that there, the energy space is (topologically equivalent to) $H^1(\Omega)$ while in the Landweber iterations we choose to work in $L^2(\Omega)$. We could do Landweber iterations considering $\L$ as an operator from $H_0^1(\Omega)$ to $H^1((0,T)\times\Gamma)$, as well or even as an operator from $H_0^1(\Omega)$ to $L^2((0,T)\times\Gamma)$. In fact, we tried the latter numerically. Then $\L$ is a smoothing operator of degree $1$ (an elliptic FIO of order $-1$ associated with a local diffeomorphism), and this makes the problem always unstable in those spaces. The reconstructions we get are very blurred as one would expect. 

Finally, if one wants to achieve some regularizing effect with a regularizing parameter, similar to Tikhonov regularizations, one could introduce the following $H^1$ norm
\[
\|f\|_{H^1_\eps}^2 = \int_{(0,T)\times\Omega }\left( |f|^2c^{-2} +\eps |\nabla f|^2\right) \d x,
\]
where $\eps>0$ is a parameter. Then we think of $\L$ as the operator $\L:H_\eps^1(\Omega)\to L^2((0,T)\times\bo)$, restricted to functions supported in  $\Omega_0$. 
 This replaces $\L^*$ in the inversions by $(1-\eps c^2\Delta_D)^{-1}\L^*$, where $\Delta_D$ is the Dirichlet Laplacian in $\Omega$. Another modification would be to keep $L^2(\Omega,c^{-2}\d x)$ but to replace $L^2((0,T)\times\bo)$ by a certain $H^{-1}$ space with a parameter. This would put a low pass frequency filter to the right of $\L^*$ rather than to the left. One such  reconstruction is shown in Figure~\ref{unstable_noise_reconstruction}.

The plan of this work is the following. We review the Landweber iterations in section~\ref{sec_Land} and study the convergence or the divergence, the dependence on the step $\gamma$, etc. In section~\ref{sec_TAT}, we compute the adjoint of the measurement operator $\L$ with full or partial data. We present numerical examples in section~\ref{sec_numeric}. A review of the averaged time reversal method is included in section~\ref{sec_ATR}. 

\textbf{Acknowledgments.} The authors thank Francois Monard and Jianliang Qian for their advice during the preparation of this work. Reference \cite{Zuazua-numerics} was pointed out to the authors by Francois Monard.

\section{Landweber iterations}\label{sec_Land} 
As explained in the Introduction, one of the goals of this paper is to analyze further the Landweber iterations method for linear inverse problems, its relation to the stability or the instability of the problem and to the power spectrum of the function we want to recover (defined as the differential $\d\mu_f(\lambda)$ of $\|P_\lambda f\|^2$, where $P_\lambda$  is the spectral projection  of $\L^*\L$). We also study the effect of the choice of the step $\gamma$ in the iterations and how they are influenced by noise or data not in the range. We do not study regularized versions of the iterations for severely ill posed linear problems, or non-linear problems. We do not study stopping criteria and their effect on the error, either. 
For analysis of the linear Landweber method, see, e.g.,  \cite{Vainikko_80, VainikkoV_book,Hanke_91, Byrne06, Kirsch_book} and the references there. For the non-linear one, see, e.g., \cite{HankeNS_95}. Some of the analysis below can be found in the literature, for examaple in \cite{Hanke_91}, but we did not find all of it in the literature. 

\subsection{Landweber Iterations}\label{sec_L1} 
Let $\L:\mathcal{H}_1\to \mathcal{H}_2$ be a bounded operator where $\mathcal{H}_{1,2}$ are Hilbert spaces. We want to solve the linear inverse problem of inverting $\L$. Assume no noise first.  Then we write 
\be{Pf}
\L f=m
\ee
in the form
\be{P0}
(I - (I-\gamma \L^* \L))f = \gamma  \L^* m.
\ee
The operator $K := I-\gamma \L^* \L$ is self-adjoint with spectrum in the interval 
\begin{equation}\label{P1}
\left[1-\gamma \|\L\|^2, 1-\gamma\mu^2 \right],
\end{equation}
where $\mu\ge 0$ is any of the stability constants in the stability estimate
\be{Sstab}
\mu \|f\|\le \|\L f\|,
\ee
which might be zero if there is no stability or even injectivity. We   choose $\mu$ to be the largest number with that property, i.e.,  $\mu^2$ is the  bottom of the (closed) spectrum   of $\L^* \L$. If $\L^*\L$ has a discrete spectrum, then $\mu$ would be the smallest singular value. In other words, $\|\L\|$ and $\mu$ are the sharp constants for which
\be{Lc}
0\le \mu^2 \le \L^* \L \le \|\L\|^2.
\ee
Then \r{P1} is the smallest closed interval containing the spectrum $\spec(K)$ of $K$ but the spectrum itself could have gaps.

By \r{P1}, we see that $K$ is a strict contraction ($\|K\|<1$) if and only if  
\be{P2}
  -1< 1-\gamma \|\L\|^2,\quad 1-\gamma \mu^2  <1,
\ee
i.e.,
\be{P3}
0<\gamma<\frac{2}{\|\L\|^2}, \quad 0<\mu.
\ee
In other words, $\gamma$ has to be small enough and the problem has to be stable. Then $f$ can be reconstructed by the following Neumann expansion
\be{S4}
\sum_{j=0}^\infty K^j\gamma\L^*  m = \sum_{j=0}^\infty ( I-\gamma \L^* \L )^j \gamma \L^* m
\ee
which converges uniformly and exponentially because 
\[
\|( I-\gamma \L^* \L )^j\L^* m\|\le \|K\|^j \|\L^*m\|. 
\]
This implies the following scheme. 

\textbf{Landweber Iterations.} Set
\be{Land}
\begin{split}
f_0 = & 0, \\
f_k = & f_{k-1} - \gamma \mathcal{L}^\ast (\mathcal{L}f_{k-1} - m), \quad\quad k=1,2,\dots .
\end{split} 
\ee
The scheme is used even when the data $m$ is ``noisy'', i.e., when $m\not\in \Ran\L$, or when there is no uniqueness. 
Sometimes, the following criterion is used. 
Let $C>1$ be a prescribed constant. The iteration is terminated when the condition
$$\| \mathcal{L} f_k - m\|_{\mathcal{H}_2} < C\delta $$
is violated for the first time, where $\delta$ is an a priori bound of the noise level, i.e., a constant for which $\|m-\L f\|\le \delta$, where $f$ is the unknown function we want to reconstruct and we think of $m$ as a noisy measurement.  
Then we take $f_k$ for such a $k$ as an approximation of $f$. If  there is stability, i.e., when $\mu>0$, this implies  $\|f_k-f\|\le C\delta/\mu$. 

It is well known that Landweber Iterations is just a gradient decent method for the functional
\[
f\longmapsto \|\L f-m\|^2_{\mathcal{H}_2}.
\]
This makes some of the properties we describe below more geometric; for example if $m\not\in\Ran\L$, a minimizer, if exists, would be the same as a minimizer for $m$ replaced by $m$ projected on $\overline{\Ran\L}$. 

\begin{figure}[h!] 
  \centering
  \includegraphics[trim = 10mm 15mm 10mm 0mm, clip, scale=0.5]{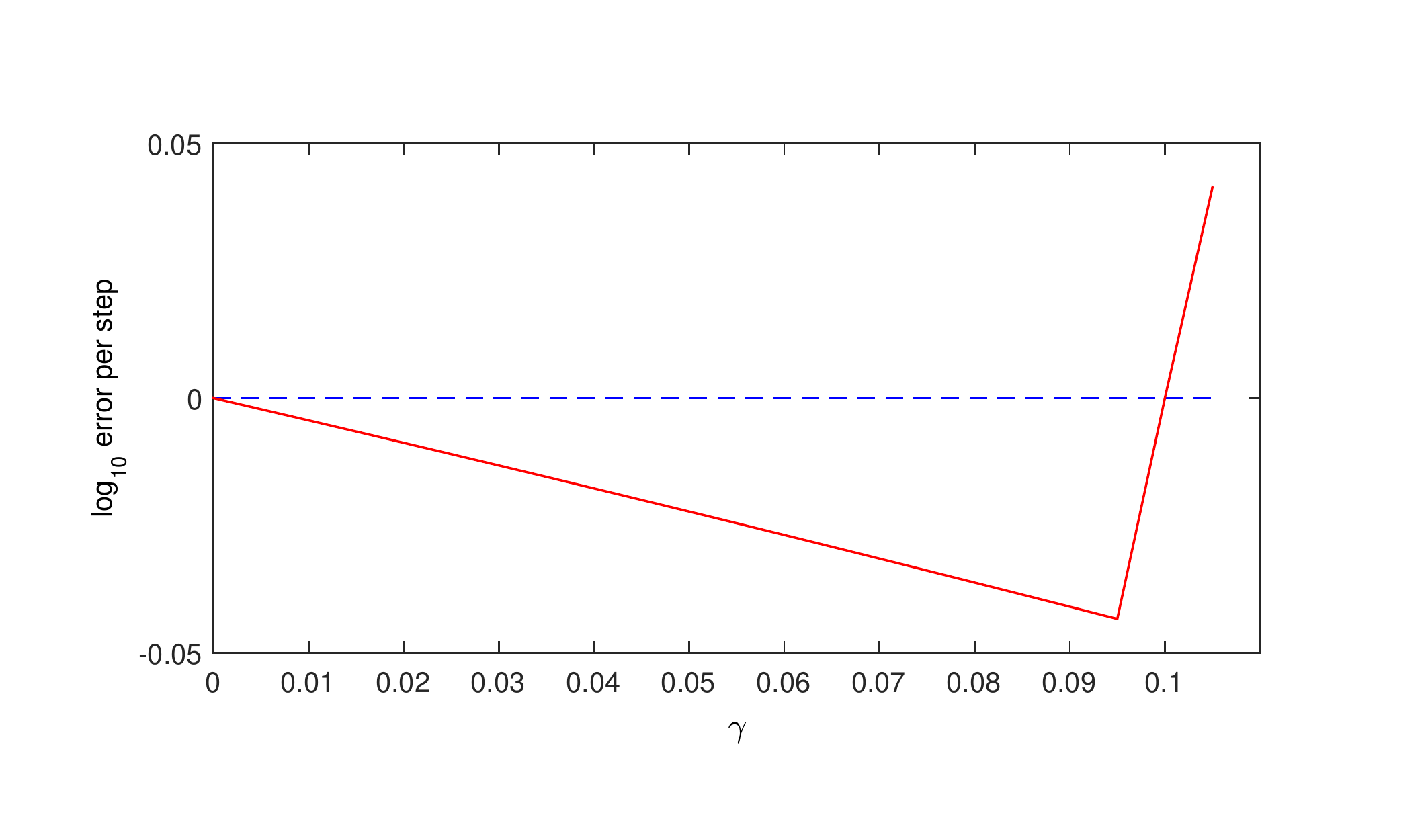}
\caption{\small $\log_{10}$ error at each step vs.\ $\gamma$ with exact data.  This is the graph of $\log_{10}\| K\|$ as a function of $\gamma$ for $\mu^2=1$, $\|\L\|^2=20$. The optimal $\gamma^*$ is $\gamma^*=2/21\approx0.0952$. The error increases fast on the right of $\gamma^*$.   }
\label{fig_theoretical_error}
\end{figure}

Let us denote $K$ by $K_\gamma$ now. To estimate the rate of convergence, we need to estimate how close $\|K_\gamma\|$ is to $1$. By \r{P1} and \r{P2},
 \be{norm_K}
\|K\|= \max\left\{ \left|1-\gamma \|\L\|^2\right|, |\gamma\mu^2-1|\right \} = \max\left\{  |1-\gamma \|\L\|^2|, 1-\gamma\mu^2\right \}.
\ee
In Figure~\ref{fig_theoretical_error}, we plot $\log_{10}\|K\|$ for $\mu^2=1$, $\|\L\|^2=20$ for a range of $\gamma$'s. This gives us the asymptotic rate of convergence in the uniform norm, since the error if we use the $N$-th partial sum is $(I-\|K\|)^{-1}\|K\|^N$, which $\log_{10}$ is $N\log_{10}\|K\|-\log_{10}(I-\|K\|)$. In Figure~\ref{fig_errors_10_30_50}, we present a numerical evidence of that behavior. Note that in the analysis below, with exact data, the factor $(I-\|K\|)^{-1}$ is actually removed, see \r{P8}, for example.

We have
\be{P4}
1-\|K_\gamma\| = \min\left(\gamma\mu^2, 2-\gamma\|\L\|^2\right)=: \nu. 
\ee
Since the expressions in the parentheses are respectively an increasing and a decreasing function of $\gamma$, the maximum is achieved when they are equal, i.e., for $\gamma$ equal to 
\be{P5}
\gamma^* = \frac{2}{\mu^2+\|\L\|^2}.
\ee
With $\gamma=\gamma^*$, 
\be{P5c}
1-\|K_{\gamma^*}\| = \frac{2\mu^2}{\mu^2+\|\L\|^2} = :\nu^*,
\ee
which is well known, see e.g.,  \cite[pp.66-67]{Byrne06}. 
Then the series is dominated by the geometric series $C\sum (1-\nu^*)^j$ which converges exponentially, and therefore, \r{S4} converges also uniformly (in the operator norm).  When $0<\nu^* \ll1$ however, the speed of convergence decreases  and the sensitivity to computational errors increases. 

For general $\gamma\in (0,2/\|\L\|^2)$, the Landweber iterations are dominated by $C\sum (1-\nu)^j$ (and this is sharp), see \r{P4}. Note that $\nu$ as a function of $\gamma$ is piecewise linear with a maximum at $\gamma^*$, indeed but since the slope of the second factor in \r{P4} is larger (and typically, much larger) by absolute value (which is $\|\L\|^2$) than the slope of the first one, $\mu^2$, increasing $\gamma$ form some very small value to $\gamma^*$ would improve the convergence gradually but after passing $\gamma^*$, the convergence will deteriorate very fast. We observed this behavior in our tests. Note that the ratio of the two slopes is the condition number $\|\L\|^2/\mu^2$ of $\L^*\L$. 

If $\mu=0$, then $\|K\|=1$. We can use the fact however that $K^j$ is applied to $\L^*m$ only.
Then $f$ can be reconstructed by the following formula, see \r{S4}, 
\be{P6}
\sum_{j=0}^\infty (I-\gamma \L^*\L)^j \gamma \L^*\L f
\ee
(with $\L f=m$ given). A zero eigenvalue of $\L^*\L$ would actually create a series of zeros, that still converges. Below, we actually allow $\L$ to have a non-trivial kernel $\Ker\L=\Ker\L^*\L$ and study $f$'s modulo that kernel. 

\begin{proposition}\label{pr_1} The following statements are equivalent. 

(a) \r{Sstab} holds  with $\mu>0$ for $f\perp\Ker\L$,

(b) $\spec(\L^*\L)$ has a gap $(0,\mu)$ with some $\mu>0$,

(c) $\Ran \L$ is closed. 
\end{proposition}

The proof is a standard application of functional analysis, see \cite{SU-JFA09}, for example.

\begin{definition}
We will call the problem stable,  if  either of the conditions above  holds.
\end{definition}

We always assume that $\mu$ is the maximal constant (which exists) for which Proposition~\ref{pr_1}(b) holds, when it does.  Such a condition is satisfied, if $\L^*\L$ is an elliptic \PDO, see for example \cite{SU-JFA09}.   

When the problem is stable, 
\be{PL}
\L :(\Ker\L)^\perp \longrightarrow \Ran\L
\ee
is invertible with a norm of the inverse equal to $1/\mu$. 
 
   A more detailed analysis is  done by the spectral theorem below. Let  
\[
S_N f = \sum_{j=0}^{N-1} (I-\gamma \L^*\L)^j\gamma \L^*\L f
\]
be the $N$-the partial sum in \r{P6}.

\textbf{Modifications.} If $f$ belongs to a subspace of $\mathcal{H}_1$, we can replace $\L$ by $\L\Pi^*$, where $\Pi$ is the orthogonal projection to that space ($\Pi\Pi^*=I)$ and then $\L^*\L$ is replaced by $\Pi\L^*\L\Pi^*$ and still apply the scheme. Also, for every bounded non-negative operator $P$ in $\mathcal{H}_2$, we can replace $\L^*\L$ by $\L^*P\L$. If $\mathcal{H}_2$ is an $L^2$ space with respect to some measure, a different choice of the measure would insert such a $P$ there. In that case, $P$ can be a   multiplication by a non-negative function, which can be used to satisfy compatibility conditions in the case we study, for example. 

\subsection{Exact data}\label{sec_exact}
\begin{theorem}[Inversion with exact data]\label{thm_LC}
Assume 
\be{Pgamma}
0<\gamma<2/\|\L\|^2.
\ee 

(a) 
Then for every $f\in \mathcal{H}_1$, the series \r{P6} converges to the orthogonal projection $f_1$ of $f$ to $(\Ker \L)^\perp$.  

(b) The series \r{P6} converges uniformly (i.e., in the operator norm) if and only if the problem is stable. In that case, it converges to   $f_1$   exponentially fast in the operator norm; more precisely, with $\nu$ as in \r{P4},
\be{Pnu}
\|S_Nf-f_1\|\le (1-\nu)^N\|f_1\|.
\ee
\end{theorem}  

\begin{remark}
As we showed above, $\nu$ is maximized for $\gamma=\gamma^*$. Then $\nu=\nu^*$, and  
\[
\|S_Nf-f_1\|\le \left(\frac{\|\L\|^2-\mu^2}{\|\L\|^2+\mu^2}\right)^{N}\|f_1\|.
\]

\end{remark}

\begin{proof}
Let 
\[
|\L| = \int_{\spec(|\L|)} \lambda\,\d P_\lambda
\]
be the spectral decomposition of $|\L|:=(\L^*\L)^{1/2}$, where $\d P_\lambda$ is the spectral measure supported on   $\spec(|\L|) = \spec(\L^*\L)$, which is contained in $[0,\|\L\|]$. In that representation, the series \r{P6} consists of the terms $(1-\gamma\lambda^2)^j\gamma \lambda^2$.

Recall that if $|\L|$ is an $N\times N$ matrix, then $P_\lambda = \sum_{\lambda_j\le\lambda}P_{\lambda_j}$, where $\lambda_j$ and $P_{\lambda_j}$ are the eigenvalues of $|\L|$ and the projections to the corresponding eigenspaces, respectively. Then $\d P_\lambda= \sum_{\lambda_j\le\lambda} P_{\lambda_j}\delta(\lambda-\lambda_j)\,\d\lambda$, and for any $f\in \R^N$, the spectral measure $\d\mu_f$ of $f$ we use below is given by $\d\mu_f = \d (P_\lambda f,f) =  \sum_{\lambda_j\le\lambda} \|P_{\lambda_j} f\|^2 \delta(\lambda-\lambda_j)\,\d\lambda$. In particular, for every $g\in L^2$, we have
\[
\int_{\spec(|L|)} g(\lambda) \,\d\mu_f  = \sum_{\lambda_j\le\lambda} |c_j|^2 g(\lambda_j),
\]
where $c_j$ are the scalar projections (the Fourier coefficients) of $f$ w.r.t.\ to an orthonormal basis of eigenfunctions corresponding to $\lambda_j$.

Using the identity $(1-q)(1+q+\dots+q^{N-1})= 1-q^N$, we get 
\[
S_N f = \int \sum_{j=0}^{N-1} (I-\gamma \lambda^2)^j \gamma\lambda^2 \, \d P_\lambda\, f = 
 \int    \left( 1-(1-\gamma \lambda^2)^N  \right) \d P_\lambda\, f.
\]
Let $f=f_0+f_1$ be the orthogonal decomposition of $f$ with $f_0\in \Ker \L= \Ker|\L|$ and $f_1\perp\Ker |\L|$. Clearly, $S_Nf_0=0$. For $S_N f_1$ we have the same formula as above
with $\d P_\lambda$ replaced by $\d P_\lambda^1$, where the latter is the restriction of $\d P_\lambda$ to the orthogonal complement of $\Ker |\L|$. On that space, $\lambda=0$ is not in  the pure point spectrum $\spec_\text{pp}(|\L|) = \spec_\text{pp}(\L^*\L)$ but may still be in the spectrum. We refer to \cite{Reed-Simon3} for the definition and the basic properties of the pure point and continuous spectra of bounded self-adjoint operators. 
Then
\be{P8}
\|S_N f-f_1\|^2 = \int_{\spec(|\L|)}  |1-\gamma \lambda^2|^{2N} \d\mu_{f_1},
\ee
where $\d\mu_{f_1}= \d(f_1, P_\lambda f_1)$ is the spectral measure associated with $f_1$. The integrand converges to  $g(x)= 0$  for $x\not=0$ and $g(0)=1$, as $N\to\infty$. It is bounded by the integrable function $1$, see \r{P2}. By the  Lebesgue dominated convergence theorem, the limit in \r{P8}, as $N\to\infty$, is $\int_{\spec(|\L|)}g\, \d\mu_{f_1}$ and the latter is zero because $\mu_{f_1}(\{0\})=0$. 

We will provide a more direct proof of the latter revealing better the nature of the problem. We claim  that
\be{P9}
\lim_{\delta\to 0+}\int_{\spec(|\L|)\cap [0,\delta)}  \d\mu_{f_1}=0. 
\ee
Indeed, as $\delta\searrow0$, the integral above (which is just the measure of $\spec(|\L|)\cap [0,\delta)$ with respect to $\mu_{f_1}$) is a monotonically decreasing non-negative function of $\delta$ and therefore, it has a limit which is its infimum. By the properties of Borel measures, if the limit is positive, that means that $\{0\}$ has a positive measure.  Therefore, $0$ would be  in the pure point spectrum which we eliminated by restricting $\L$ on the orthogonal complement of $\text{Ker}(|\L|)$. Therefore, the limit in \r{P9} is zero, as claimed. 

Choose $\eps>0$ now. Let $\delta>0$ be such that the l.h.s.\ of \r{P9} does not exceed $\eps$. Then by \r{P8}, 
\be{P10}
\begin{split}
\|S_N f-f_1\|^2 &= \int_{\spec(|\L|)\cap [0,\delta)}  |1-\gamma \lambda^2|^{2N} \d\mu_{f_1} + \int_{\spec(|\L|)\cap [\delta,\infty)}  |1-\gamma \lambda^2|^{2N} \d\mu_{f_1}\\
 &\le \eps + \int_{\spec(|\L|)\cap [\delta,\infty)}  |1-\gamma \lambda^2|^{2N} \d\mu_{f_1}. 
\end{split}
\ee
For $\lambda\ge\delta$ in $\spec(|\L|)$, the term $|1-\gamma \lambda^2|$ is bounded away from $1$, and therefore tends to $0$ exponentially fast, as $N\to\infty$. Choose $N\gg1$ so that the last integral in \r{P10} does not exceed $\eps$. Then $\|S_N f-f_1\|^2\le 2\eps$ for $N\gg1$, which completes the proof of (a).  

This argument shows that the rate of convergence depends heavily on the spectral density (measure) of $f$ near $\lambda=0$ for $\lambda>0$ but not for $\lambda=0$. 

To prove (b), notice first that the norm of $S_N-I$ on $(\Ker\L)^\perp$ is given by the maximum of $(1-\gamma\lambda^2)^N$ on $\spec(|\L|\Big |_{(\Ker\L)^\perp })$, see \cite[Theorem~VII.1(g)]{Reed-Simon3}. For uniform convergence, we need that maximum to be strictly less than $1$, which proves the first part of (b). When this happens, then $(1-\gamma\lambda^2)^N\le (1-\nu)^N$.
\end{proof}

 \subsection{Noisy data} 
Let the measurement $\L f$ be ``noisy'', i.e., we are given $m\in \mathcal{H}_2$ not in the range of $\L$ and want to ``solve'' $\L f=m$. A practical solution of such problem is to find some kind of approximation of the actual $f$ assuming that $m$ is close in some sense to the range. We will study what happens if we solve \r{P0} with a Neumann series without assuming \r{Pf}. 

\begin{theorem}[Noisy data]\label{thm_noisy} 
Assume \r{Pgamma}. Then the Neumann series \r{S4} converges for every $m\in \mathcal{H}_2$ if and only if the problem is stable. 
Moreover,

(a) If the problem is stable, 
 $\L^*\L$ is invertible on  $(\Ker\L)^\perp$  with a norm of the inverse $1/\mu^2$. Then the limit is the unique solution of $\L^*\L f=\L^* m$ (note that $\L^*m\in (\Ker\L)^\perp$) orthogonal to $\Ker\L$. 
The convergence is uniform with a rate (see  \r{Pnu}) $(1-\nu)^N/\mu$.

(b) If the problem is not stable, then there is no uniform convergence. Moreover, there is no even strong convergence. The set of $m$ for which \r{S4} is unbounded (and hence, diverges) is residual and in particular, dense,  i.e., its complement is of first Baire category. 
\end{theorem}

\begin{proof}
Recall that every bounded operator $\L$ has a polar decomposition of the type  $\L= U|\L|$, see \cite[VI.4]{Reed-Simon1}, true also for bounded operators between different Hilbert spaces, where  $U$ is a partial isometry. We have  $\Ker U = \Ker \L$, $\Ran U= \overline{\Ran \L}$, 
and $U$ is an isometry on $(\Ker \L)^\perp$.  We can write \r{P0} as 
\be{P11}
(I - (I-\gamma \L^* \L))f = \gamma  |\L|U^* m.
\ee
For $U^*$, we have $\Ker U^*= (\Ran \L)^\perp$ and ${\Ran U^*} = (\Ker\L)^\perp$. Set $m_*= U^*m$. We have  $m_*\perp \Ker\L$. 
Then \r{P11} reduces to
\be{P12}
(I - (I-\gamma \L^* \L))f = \gamma  |\L|m_*.
\ee
For
\[
\tilde S_N m := \sum_{j=0}^{N-1} (I-\gamma \L^*\L)^j\gamma  |\L|m_*
\]
we now have 
\[
\begin{split}
\tilde S_N m &= \int_{\spec(|\L|)} \sum_{j=0}^{N-1} (I-\gamma \lambda^2)^j \gamma\lambda  \, \d P_\lambda m_*\\  
 &= \int_{\spec(|\L|)}   g_N(\lambda) \, \d P_\lambda\, m_*,\qquad  g_N(\lambda) :=  \frac{ 1-(1-\gamma \lambda^2)^N}\lambda.
\end{split}
\]
The function $g_N$ is smooth even near $\lambda=0$. Therefore, in the spectral representation,  $\tilde S_N= g_N(\lambda)$.

To prove (a), assume  that the problem is stable. Then $\L^*\L$ is invertible on  $(\Ker\L)^\perp$   by the spectral theorem. The function $1/\lambda$ 
is a bounded function on the spectrum, and it is the spectral representation of the inverse of $|\L|$ restricted to $(\Ker\L)^\perp$ that we will denote temporarily by $|\L_1|^{-1}$. Notice that $|\L_1|^{-1}U^* = \L_1^{-1}$, where $\L_1:(\Ker\L)^\perp\to \text{Ran}\,\L$ (the latter is closed) is  the restriction of $\L$ as in \r{PL}. Therefore, denoting by $m_1$ the orthogonal projection of $m$  to $(\Ker\L)^\perp$, we have 
$\L_1^{-1}m_1= |\L_1|^{-1}U^*m=|\L_1|^{-1}m_*$ which is the spectral representation is $\lambda^{-1}m_*$. 

Therefore,
\[
\tilde S_N m-\L_1^{-1}m_1 =  -\int_{\spec(|\L|)}   \frac{(1-\gamma \lambda^2)^N}\lambda \, \d P_\lambda\, m_*.
\]
Note that $f=\L_1^{-1}m_1$ is the unique solution of $\L^*\L f=\L^*m$ orthogonal to $\Ker\L$. Then
\be{Per}
\big\|\tilde S_N m-\L_1^{-1}m_1\big\|^2 =  \int_{\spec(|\L|)}   \frac{(1-\gamma \lambda^2)^{2N}}{\lambda^2} \, \d \mu_{m_*}.
\ee
The convergence then follows as in Theorem~\ref{thm_LC} because then the integrand is uniformly bounded on the spectrum.

Assume now  that the problem is unstable. 
Assume that $\tilde S_N m$ is  bounded (which would be true if it converges) for every $m_*\in (\Ker\L)^\perp$ with a bound that may depend on $m_*$. By the uniform boundedness theorem, the norms $\|\tilde S_N\|$ must be uniformly bounded, as well. Therefore, by \cite[Theorem~VII.1(g)]{Reed-Simon3}
\be{P13}
\|\tilde S_N\|=g_N(\lambda) \le C, \quad \forall N, \; \forall \lambda\in \spec(|\L|).
\ee
On the other hand, $g_N(N^{-1/2}) =\sqrt{N}(1-e^{-\gamma}+o(1))$, as $N\to\infty$.   
This contradicts \r{P13} since in the non stable case, $\spec(|\L|)$ contains a sequence of positive numbers converging to $0$; and in particular proves lack of uniform convergence.  We also proved that there cannot be strong convergence in this case. 
 
\begin{figure}[h!] 
  \centering
  \includegraphics[trim = 20mm 208mm 100mm 18mm, clip, scale=0.8]{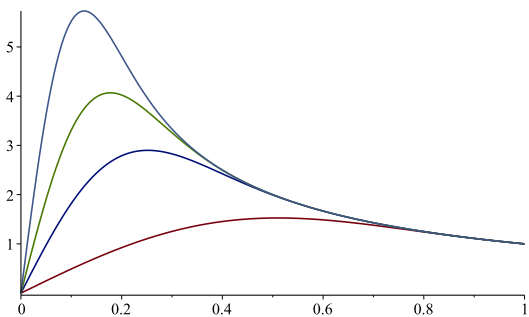}
\caption{\small The functions $g_N(\lambda)$ with $\gamma=1$ and $N=5, 20, 40, 80$. As $N$ increases, the maximum increases as $C_1\sqrt N$ and its location shifts to the left to $\lambda_N=C_2/\sqrt N$.}
\label{TAT_with_reflection_numerics_fig1}
\end{figure}

The proof is then completed by noticing that the set of $m$ for which the $\tilde S_N m$ is unbounded is known to be residual since $\|\tilde S_N\|$ is unbounded. Residual sets are called sometimes generic. 

We give a more direct proof of the last statement, revealing the structure of those $m$ for which the iterations diverge. We will show in particular, that the span of such $m$ is infinitely  dimensional. 
Notice first that  $\spec(\L^*\L)$ contains a sequence $\lambda_n\searrow 0$. 
Let $I_k=(\lambda_k,\lambda_{k+1}]$. By the standard proof that $\mathbf{N}\times\mathbf{N}$ is countable, where $\mathbf{N}=\{1,2,\dots\}$ (or by its conclusion), we can label $I_k$ and $\lambda_k$ as $I_{l,m}$ and $\lambda_{i,m}$ so that the map $k\to(l,m)$ is a bijection. For each $l$, $\lambda_{l,m}\to0$, as $m\to\infty$ because the latter is a subsequence of $\lambda_k$ with rearranged terms. Then for every $l$, on the range $\mathcal{H}_l$ of the spectral projection corresponding to $\cup_k I_{l,k}$, the series \r{S4} diverges for some $m$ so that $U^*m\in \mathcal{H}_l$. On the other hand, $U$ is unitary from $(\Ker \L)^\perp$ to  $\overline{\text{Ran}\,\L}$; therefore, $U^*: \overline{\text{Ran}\,\L} \to (\Ker\L)^\perp$ is unitary, as well. Since $\mathcal{H}_l\subset (\Ker\L)^\perp$, $U^*m\in \mathcal{H}_l$ identifies $m\in \overline{\text{Ran}\,\L}$ uniquely by applying the inverse of $U^*$. This proves the existence of $m=m_l$ which makes \r{S4} divergent, and those $m_l$'s are mutually orthogonal. They are mapped to a linearly independent set under $U^*$. 

Note that we actually proved something more: for every $l$, ``almost all'' elements $m$ in $\mathcal{H}_l$ force \r{S4} to diverge: only those with finitely many non-zero projections do not. On the other hand, there is an infinitely dimensional space of $m$ for which we have strong convergence, for example all $m$ with a spectral measure $\d\mu_{m_*}$ such that $1/\lambda$ is integrable. 
\end{proof}

\begin{remark}
Since 
\[
N^{-1/2}g(x/\sqrt{N})= \frac{1-(1-\gamma x^2/N)^N}{x} \sim \frac{1-e^{-\gamma x^2}}{x}, \quad \text{as $N\to\infty$}, 
\]
it follows that $\max_{0\le \lambda\le \sqrt{2/\gamma}} |g_N|\sim C\sqrt{N}$ attained at $\lambda\sim  C_1/\sqrt N$. 
\end{remark}

\begin{remark} 
It is worth noting that in the stable case (a), the Landweber solution $f=\L_1^{-1}m_1$ can also be described in the following way. The noisy data is first projected to  $\Ran\L$ (which is closed). Denote that projection by $m_1$. Then the exact solution is found as in the case of exact data: the solution of $\L f=m_1$ orthogonal to $\Ker\L$. This is called sometimes the Moore-Penrose inverse. 
 If there is no stability, $m$ is projected to the closure  of $\Ran\L$, because in $m_*=U^*m$ the result would be the same if we replace $m$ by $m_1$. Then we take  $m_*=U^*m_1$. Then \r{P11} is equivalent to $|\L|^2f= |\L|m_*$. This is an equation in $(\Ker\L)^\perp$, equivalent there to $|\L|f=m_*$. We have $m_*\in (\Ker|\L|)^\perp$ and the later is the closure of $\Ran|\L|$ but it is guaranteed to be in range if and only if that range is closed, i.e., if there is stability. 
\end{remark}

\begin{remark}
If $\mathcal{H}_1$ is finitely dimensional, then the problem is always stable. If we discretize a problem in an infinitely dimensional space, we get a finitely dimensional one. One may think that there is no convergence problem then. There are two potential problems with this argument however. First, the gap could be very small causing stability problems in numerical inversions. Second, the discretization could be a very poor approximation of the continuous problem near the bottom of the spectrum causing a very small gap (optimal $\mu$) even though the continuous model may have a not such a small one.  In our simulations, for example,  we get eigenvalues of order $10^{-15}$ while the ``significant ones'' are of order $1$. 
\end{remark}

\begin{remark}
The proof of Theorem~\ref{thm_noisy} reveals that the rate of convergence or possible convergence with noisy data depends on the spectral measure of $m^*$, which we defined as the orthogonal projection of $U^*m$ onto $(\Ker\L)^\perp$. 
\end{remark}

\subsection{Choosing a  different initial guess} We may think $f_0$ in \r{Land} as an initial guess. We may also think of $f_1= \gamma\L^*\L m$ as such. When  $\|K\|<1$, we have a sequence generated by a contraction map, and this sequence would converge regardless of the initial $f_0$. We then choose $f_0$ not necessarily zero and set, as before,
\[
f_k =  f_{k-1} - \gamma \mathcal{L}^\ast (\mathcal{L}f_{k-1} - m), \quad\quad k=1,2,\dots .
\]
Then the limit $f$, if exits,  must satisfy 
\[
f =  f - \gamma \mathcal{L}^\ast (\mathcal{L}f - m),
\]
which is equivalent to $\mathcal{L}^\ast \mathcal{L}f = \L^* m$, which is the starting point in Section~\ref{sec_L1}, see \r{P0}. If $f_0$ is already close enough to $f$, we would expect the iterations to converge faster. In fact, we could gain speed if $f-f_0$ has a better power spectrum w.r.t.\ $\L^*\L$, i.e., has a smaller spectral measure there. 

It is easy to show that 
\[
f_{N} = (1+K+\dots+K^{N-1})\gamma \L^*m+ K^Nf_0. 
\]
Therefore,
\[
\begin{split}
f_N-f_0&= (1+K+\dots+K^{N-1})\gamma \L^*m+ (K^N-I)f_0\\
  & =(1+K+\dots+K^{N-1})(\gamma \L^*m+(K-I)f_0) \\
  & =  (1+K+\dots+K^{N-1})\gamma \L^*(m- \L f_0).
\end{split}
\]
Therefore, the Landweber iterations in this case correspond to the ones we studied so far for the equation
\[
\L^*\L(f-f_0) = \L^* (m- \L f_0),
\]
i.e., $f$ is replaced with $f-f_0$ and $m$ is replaced by $m-\L f_0$. Even though the new equation is equivalent to the old one, if $f-f_0$ is small in a certain sense or it has a lower spectral measure near the origin, the series may converge faster. 
 
\section{Application to multiwave tomography with reflectors}\label{sec_TAT}

We apply the abstract theory in the previous section to  multiwave tomography with reflectors. To have $\L$ as a  bounded $L^2\to L^2$ map, we will restrict $\supp f$ to a compact subset of $\Omega$, and assume  convexity. The reason for that is to exclude possible singularities issued from $\supp f$ hitting $\bo$ tangentially. We also refer to \cite{Tataru_98} for a discussion of the mapping properties of a similar operator (the solution of the wave equation without reflections, i.e., propagating in $\R^n$) to $\R\times\bo$ in case of tangential rays.

First we derive the expression for  
\be{L*}
\mathcal{L}^\ast: L^2((0,T)\times\partial\Omega) \rightarrow  L^2(\Omega_0,c^{-2}\d x).
\ee
 We  know that $\L^*$ must exist as a bounded operator in the spaces indicated above, and  $C_0^\infty((0,T)\times\partial\Omega)$ is dense in $L^2((0,T)\times\partial\Omega)$; therefore $\L^*$ restricted to the latter space defines $\L^*$ uniquely.

\begin{theorem}
The adjoint operator of $\mathcal{L}$ defined in \eqref{L} is the closure of 
\begin{equation} \label{adjL}
\L^\sharp:  C_0^\infty((0,T) \times\partial\Omega) \rightarrow L^{2}(\Omega, c^{-2}\d x), \quad\quad\quad g\mapsto -\partial_t v|_{t=0,\, x\in\Omega_0},
\end{equation}
where $v$ is the solution of the following initial boundary value problem
\begin{equation} \label{adjBVP}
\left\{\begin{array}{rcll}
(\partial^2_t - c^2(x) \Delta) v = & 0 & \quad\quad \text{ in } (0,T)\times\Omega,  \\
\partial_\nu v |_{(0,T)\times\partial\Omega} = & g,  & \\
v|_{t=T} = & 0 ,& \\
\partial_t v|_{t=T} = & 0. &
\end{array} \right.
\end{equation}
\end{theorem}

\begin{proof}
First we assume  $f\in C_0^\infty(\Omega_0)$ in \eqref{BVP} and $g\in C_0^\infty((0,T)\times\partial\Omega)$ in \eqref{adjBVP}. Note that \r{adjBVP} is uniquely solvable, and one way to construct a solution is to take a smooth $v_0$ satisfying the boundary condition in \eqref{adjBVP}, supported in $(0,T)\times\bar\Omega$ and set $v=v_0+v_1$, where $\Box v_1 = -\Box v_0$, and $v_1$ has zero Cauchy data on $t=T$ and satisfies the homogeneous Neumann boundary condition. Here, $\Box$ is the wave operator and the source problem for $v_1$ can be solved by Duhamel's principle using the well posedness of the boundary value problem for the Neumann problem with Cauchy data in $H^1(\Omega)\times L^2(\Omega)$, see e.g., \cite{Goldstein2003} or \cite{St-Y-AveragedTR} in this context. In \cite{St-Y-AveragedTR}, in the energy space, $H^1$ is replaced by the same space modulo constants (the kernel of the Neumann realization of $-c^2\Delta$) but for Cauchy data $(C,0)$ with $C$ constant, the solution is trivially $u=C$.

Let $v$ be a solution of \eqref{adjBVP}, and let $u$ solve \r{BVP}. Then
\begin{equation} \label{zero}
0=\int_\Omega \int^T_0 \left[ \frac{1}{c^2(x)}\partial^2_t v - \Delta v \right] u \,dtdx = \int_\Omega \frac{1}{c^2(x)} \int^T_0 \partial^2_t v u \,dtdx - \int_\Omega \int^T_0 \Delta v u \,dtdx.
\end{equation}
For the first term in \eqref{zero}, integration by parts in $t$ twice gives
\begin{align*}
\int_\Omega \frac{1}{c^2(x)} \int^T_0 \partial^2_t v u \,\d t\, \d x = & \int_\Omega \frac{1}{c^2(x)} \left[  \partial_t v u \Big\vert^T_0 - \int^T_0 \partial_t v \partial_t u \,\d t \right] \d x \\
 = & -\int_\Omega \frac{1}{c^2(x)} (\partial_t v|_{t=0}) f \,\d x - \int_\Omega \frac{1}{c^2(x)} \int^T_0 \partial_t v \partial_t u  \,\d t\, \d x\\
 = & -\int_\Omega \frac{1}{c^2(x)} (\partial_t v|_{t=0}) f \,\d x - \int_\Omega \frac{1}{c^2(x)} \left[  v \partial_t u \Big\vert^T_0 - \int^T_0 v \partial^2_t u \,\d t \right] \d x \\
 = & -\int_\Omega \frac{1}{c^2(x)} (\partial_t v|_{t=0}) f \,\d x + \int_\Omega \frac{1}{c^2(x)} \int^T_0 v \partial^2_t u  \,\d t\, \d x,
\end{align*}
where we have used that $u|_{t=0}=f$, $\partial_t  u|_{t=0}=0$ and $v|_{t=T} = \partial_t v|_{t=T}=0$.
For the second term in \eqref{zero}, applying Green's formula yields
$$
\int_\Omega \int^T_0 (\Delta v) u  \,\d t\, \d x = \int_\Omega \int^T_0 v \Delta u  \,\d t\, \d x + \int^T_0 \int_{\partial\Omega} g u \,\d S(x)\,\d t - \int^T_0 \int_{\partial\Omega} v \partial_\nu u \,\d S(x)\, \d t.
$$
Notice that the last term on the right-hand side actually vanishes as $\partial_\nu u=0$. Inserting these expressions into \eqref{zero} and observing that $\partial^2_t u = c^2(x)\Delta u$ we have
\begin{align*}
0 = & -\int_{\Omega_0} {c^{-2}(x)} (\partial_t v|_{t=0}) f \,\d x - \int^T_0 \int_{\partial\Omega} g u \,\d S(x)\d t \\
 = & \int_{\Omega_0} (\L^\sharp g)  f  c^{-2}\d x - \int^T_0 \int_{\partial\Omega} g \, \mathcal{L}f \,\d S(x)\,\d t,
\end{align*}
which justifies \eqref{adjL}. Since $C^\infty_0(\Omega_0)$ and $C_0^\infty((0,,T)\times\partial\Omega)$ are dense in $L^2(\Omega_0, c^{-2}\d x)$ and $L^2((0,T)\times\partial\Omega)$, respectively, the closure of $\L^\sharp$  is indeed the adjoint operator of $\mathcal{L}$.
\end{proof}

\begin{remark}
The proof of the theorem reveals the following. Note first that if $g$ is smooth but not $C_0^\infty$ (more specifically, not vanishing at $t=T$), the compatibility condition for the first derivatives of $v$ is violated at $t=T$ and $x\in\bo$ since $v|_{t=T}=0$ implies $\partial_\nu v|_{t=T}=0$ as well, which requires $g|_{t=T}=0$. In that case, there is no $C^1$ solution in the closed cylinder; and for a $C^2$ solution, we need to satisfy even one more compatibility condition: $\partial_\nu g|_{t=T}=0$.  We are interested in $\L^*\L$, and we cannot hope to have $g|_{t=T}=g_t|_{t=T}=0$ for $g=\L f$. Therefore, when computing $\L^*\L f$, we (almost) always must compute a weak solution of \r{adjBVP} and therefore deal with $g$ violating the compatibility conditions. 

The way to overcome that difficulty is to notice that we are constructing an $L^2\to L^2$ operator, see \r{L*}. Let $\chi_\eps$ be the characteristic function of $[0,T-\eps]$, with a fixed $\eps>0$. Then $\chi_\eps g\to g$ in $L^2$ for any $g\in L^2$ but the rate of convergence is $g$-dependent (for the multiplication operator, we have $\chi_\eps\to \Id$ strongly only). Therefore, if we replace $g$ by $\chi_\eps g$ and solve the resulting problem (the jump of $\chi_\eps g$ at $t=T-\eps$ is not a problem), we would get a good approximation for $0<\eps\ll1$. On the other hand, in the iterations above, we apply $\L^*$ to $g=\L f$ for a sequence of  $g$'s. One could use the results in \cite{BardosLR_control} to investigate whether $\chi_\eps g\to g$ uniformly in $g$ for $g$ in the range of $\L$ for $T>0$ satisfying the stability condition but we will not do this. Another way is just to replace $\L$ by $\chi_\eps \L$, and therefore, $\L^*$ will also be replaced by $\L^*\chi_\eps$. In fact, we could put any real valued $L^\infty$ function there, as long as it has a positive lower bound in a set on $(0,T)\times\bo$ which set set  satisfy the stability condition. 
\end{remark}

\begin{corollary} 
The adjoint $\L_\Gamma^*$ to $\L_\Gamma$ defined in \r{LG} is $E\L^*$, where $E$ is the operator of extension as zero of functions defined on $(0,T)\times\Gamma$ to $(0,T)\times\bo$. 
\end{corollary}

\begin{proof}
We have $\L_\Gamma= R\L$, where $R$ is the restriction to $(0,T)\times\Gamma$. Since $R^*=E$, this completes the proof. 
\end{proof}

\section{Numerical examples}\label{sec_numeric}
\subsection{A stable example: variable speed, whole boundary.
}
\subsubsection{Exact data}  \label{sec_5.1.1}
We start with the Shepp-Logan (SL) phantom with a variable speed $c= 1+0.3\sin(\pi x)+0.2\cos(\pi y)$ in the box $[-1,1]^2$. The speed takes values between $0.5$ and $1.5$. We take $T=4$ and observations on the whole boundary. That time is large enough to guarantee stability. 

We compare the Averaged Time Reversal method \cite{St-Y-AveragedTR} with the Landweber method. We use a standard second order finite difference scheme on an $N\!\times\!N$ grid, and we restrict the image to a smaller square by multiplying by zero in the frame staying at $3\%$ from the border. This does not affect the SL phantom but it affects $\L^*$ (and $\L$) in the iterations; we have $\chi \L^*\L\chi$ instead of $\L^*\L$ in the iterations, where $\chi$ is the cutoff function. This does not eliminate completely geodesics through $\supp f$ tangent to the boundary if $c$ is not constant but by \cite{BardosLR_control}, the stability is still preserved. 

The errors are presented in Figure~\ref{fig_errors_10_30_50}. The thinner curves in the top part are the errors on a $\log_{10}$ scale for different choices of the parameter $\gamma$. Since we use the step $\d t = \d x/(\sqrt2 \max(c))$, $\d x= \d y=2/(N-1)$ in the numerical implementation, this effectively rescales the $t$-axis by the coefficient $1.5\sqrt2$, and therefore, changes the $L^2$ norm on $[0,T]\times\bo$. The choice of $\gamma$ depends on that factor, as well.

We see in Figure~\ref{fig_errors_10_30_50} that the optimal $\gamma$ seems to be  $\gamma\approx 0.055$, based on $50$ iterations. The convergence is highly sensitive to increasing $\gamma$ beyond this level. On the other hand, decreasing $\gamma$ reduces the accuracy but this happens much slower. Choosing the optimal $\gamma$ is based on trials and errors in this case. In Figure~\ref{fig_errors_10_30_50}, we show the errors vs.\ $\gamma$ for $10$, $30$ and $50$ iterations. The errors are consistent with the theoretical analysis in Section~\ref{sec_exact}. The error curves for each fixed number of steps ($10$, $30$ or $50$) have shapes similar to that in Figure~\ref{fig_theoretical_error}. Note that to compare Figure~\ref{fig_errors_10_30_50}, and Figure~\ref{fig_theoretical_error}, one has to divide the values in the former graphs by the number of the steps. 

The theoretical analysis of the uniform convergence rate illustrated in  Figure~\ref{fig_theoretical_error}, allows us to estimate $\mu^2$ roughly. The slopes of the curves in Figure~\ref{fig_errors_10_30_50} up to the optimal $\gamma$, normalized to one step, give a slope in the range $[1,2]$. Note that it is closer to $2$ for $N=10$, and close to $1$ for $N=50$. Converting to natural log, we get a rough estimate of $\mu^2$ in $[2.3, 4.6]$. Estimating $\|L\|$ is not practical from that graph.

\begin{figure}[h!] 
  \centering
  \includegraphics[trim = 3mm 35mm 0mm 25mm, clip, scale=0.8]{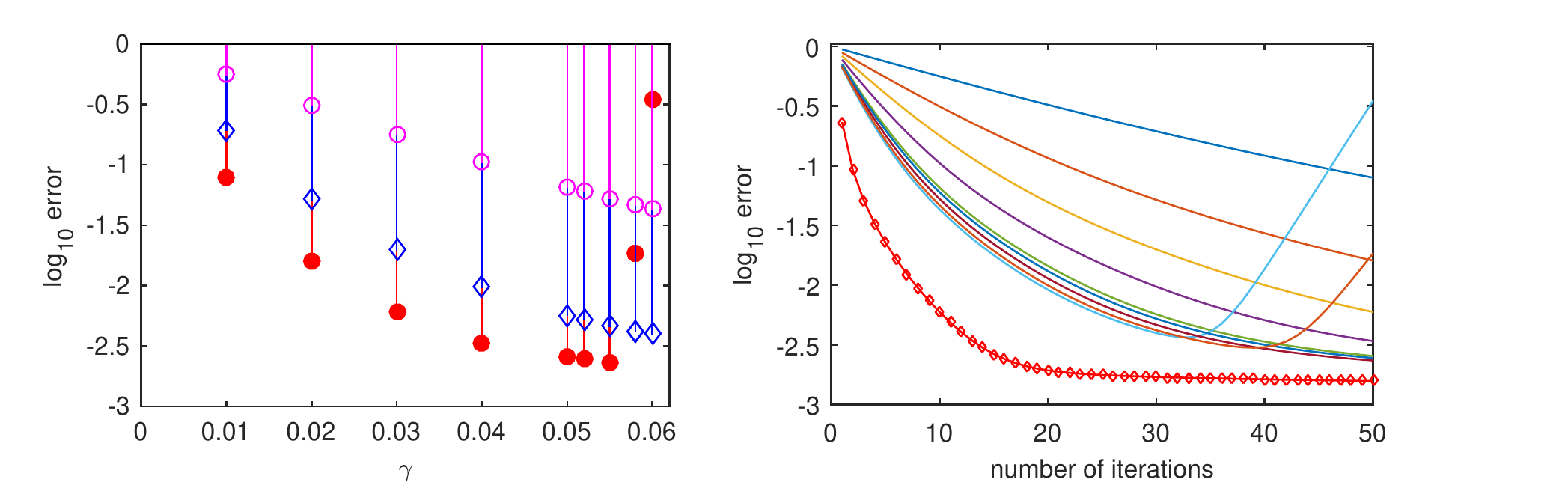}
\caption{\small 
Error vs.\ $\gamma$ for the Landweber iterations in the stable example. 
Plotted are errors after $10$ iterations (boxes), after $30$ iterations (diamonds) and after $50$ ones (dots). Right:  Error vs.\ the number of the iterations. The bottom curve with the circles are the ATR errors. The other curves correspond to $\gamma$ ranging from $0.1$ to  $0.06$, as on the left, starting form the top to the bottom in the top left corner. }
\label{fig_errors_10_30_50}
\end{figure}

We computed the matrix $L$ of $\L$ in its discrete representation for $N=101$, and $\Omega_0$ being an interior $94\times 94$ square. We halved $N$ because computing $L$ and its spectrum is very expensive for $N=201$. The errors in that case are similar to what we presented for $N=101$. 
We computed $L$ on that grid. The eigenvalues of $L^*L$ are plotted in Figure~\ref{fig_spectral_decomposition}, the smooth curve with the vertical axis on the left. 

 \begin{figure}[h] 
  \centering
  \includegraphics[trim = 10mm 0mm 0mm 40mm, clip, scale=0.6]{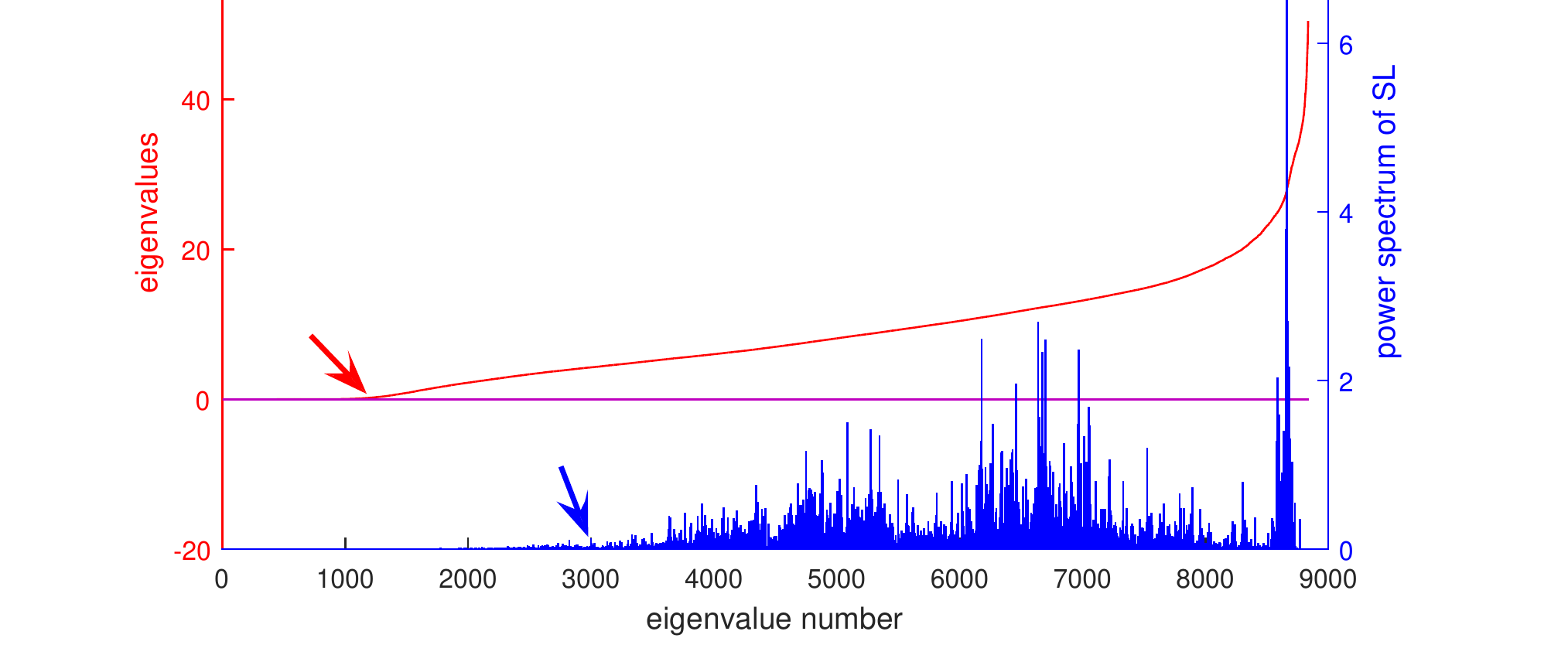}
\caption{\small A stable case. The smooth increasing curve represents the eigenvalues of $\L^*\L$ in the discrete realization on an $101\!\times\!101$ grid restricted to a $94\!\times\!94$ grid. There are small imaginary parts along the horizontal line in the middle. 
$\L^*$ is computed as the adjoint to the matrix representation of $\L$. The rough curve represents the squares of the Fourier coefficients of the SL phantom.  }
\label{fig_spectral_decomposition}
\end{figure}

We computed  the numerical representation $L_{\rm w}^*$ of $\L^*$ using the wave solver to solve \r{adjL} as well. The actual numerical inversion uses $L_{\rm w}^*L$, not $L^*L$. The spectral representation looks a bit different then but the locations of the lower part of the spectrum and the power spectrum of the SL phantom remain almost the same. Some of the eigenvalues  have imaginary parts small relative to their real ones. 

The largest computed eigenvalue of $L^*L$ is  $\lambda_{8836}\approx 50.5$, and the smallest 30 are of order $10^{-15}$, which can be thought of as $0$ computed with a double numerical precision.  We also get $\lambda_{1000}\approx 0.036$. So for all practical purposes, $L_{\rm w}^*L$ has a kernel of dimension at least $1,000$! This seems to contradict the fact that $\L$ is stable. The computed eigenfunctions of $L^*L$ (not $L_{\rm w}^*L$) are shown in Figure~\ref{fig_eigenfunctions_1_1000_7500}. The eigenfunction problem is unstable but for our purposes, it is enough to know that there is a certain approximate orthonormal (enough) base so that on the first, say $1,000$ eigenfunctions, $L$ has a norm of several orders of magnitude smaller than $\|L(SL)\|/\|SL\|$, where SL is the SL phantom. 
We verified numerically that $L$    acting on those computed eigenfunctions have  small norms, as the small eigenvalues suggest, and that the basis is orthonormal to a good precision. We also generated  linear combinations of the first $1,000$ eigenfunctions with random coefficients and computed the norm of $L$ on them. The relative norm of the SL phantom on that space is around $0.0016$, which means that the total spectral measure there is the square of that, i.e., approximately $2.42\times 10^{-6}$. 
 Also, using so generated phantoms, the iterations break. 

It turns out that the eigenfunctions corresponding to the extremely small eigenvalues have sharp jumps from minimal to maximal values from cell to cell like a chess board. The plotted $1,000$-th eigenfunction has mostly high frequency content as well, see  Figure~\ref{fig_eigenfunctions_1_1000_7500}.  The $7,500$-th one has a much lower high frequency content. This suggests that the extremely small eigenvalues are due to the poor accuracy of the wave equation solver for frequencies close to Nyquist, see also section~\ref{sec_grid}. This can also be confirmed by using semiclassical methods. The operator $\L^*\L-z$ for $0<z\ll1$ is elliptic because $z$ is separated from the range of the (elliptic) principal symbol. Therefore,  possible eigenfunctions with eigenvalues $z$   should be smooth and should also  have low oscillations. 
 \begin{figure}[h!] 
  \centering
  \includegraphics[trim = 40mm 35mm 0mm 25mm, clip, scale=0.8
  ]{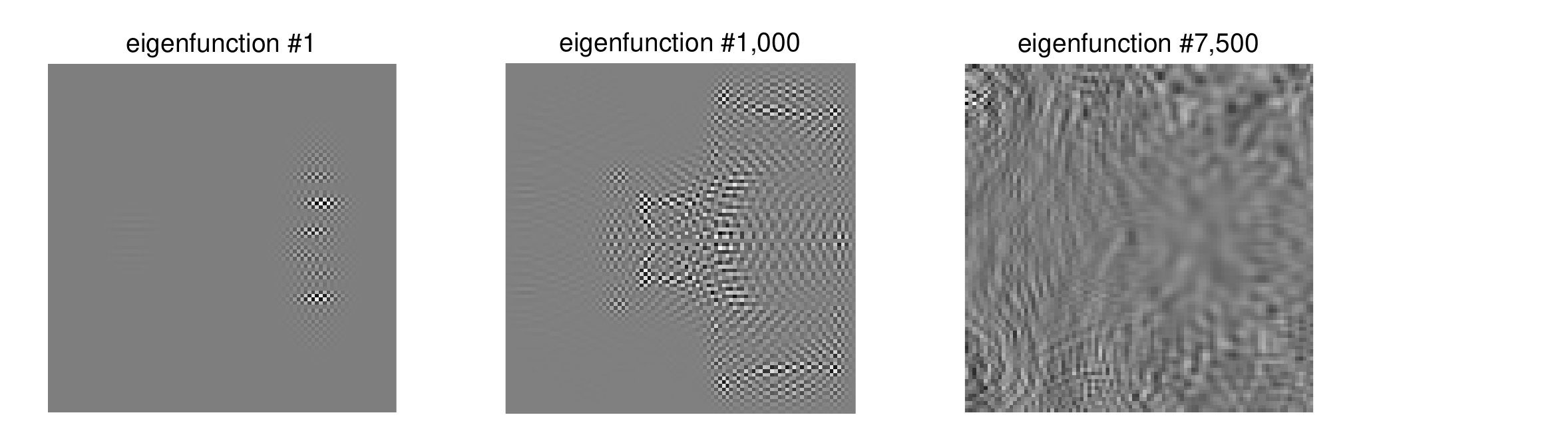}
\caption{\small The eigenfunctions of $L^*L$ corresponding to $\lambda_1$, $\lambda_{1,000}$ and $\lambda_{7,500}$ }
\label{fig_eigenfunctions_1_1000_7500}
\end{figure}
Therefore, the spectrum of the discrete $L_{\rm w}^*L$ contains small eigenvalues not approximating the spectrum of $\L^*\L$. 

Why do the iterations still work, and we are getting a lower bound of the spectrum $\mu^2$ in $[2.3, 4.6]$ (based on one $f$, indeed but we tried a few other ones)? It turns out that our $f$ (the SL phantom) does not have much spectral content in the bottom of the spectrum. In Figure~\ref{fig_spectral_decomposition}, we plot the computed  squares of the Fourier coefficients (the power spectrum) of the SL phantom with respect to the eigenfunction basis of $L_{\rm w}^*L$, see the rough curve there and the vertical axis on the right. It turns out that the SL phantom has very small Fourier coefficients until, say, $\lambda_{3,000}\approx 4.26$. 
 We can take this as an effective value for $\mu^2$ since the analysis above says that the bottom on the spectrum should be taken on the support of the spectral measure of $f$, see \r{P8}.
 This correlates well with our estimate  $\mu^2\in[2.3,4.6]$ based on Figure~\ref{fig_errors_10_30_50}. In the top of the spectrum, there is a similar but a smaller gap with the highest significant eigenvalue $\lambda_{8,718}\approx31.84$. Therefore,  $\|L\|^2\approx 31.84$ on the support of the spectral measure of the SL phantom is a good estimate. 
That gives us, roughly speaking, $\gamma^*=0.055$, which is close to our numerical results, see 
  Figure~\ref{fig_errors_10_30_50}. We recall that the computed errors are based on inversion with $L_{\rm w}^*L$ (even though we do not compute the matrices $L_{\rm w}^*$ and $L$ in the inversion) while the spectral representation is done based on $L^*L$.

It is well known that finite difference schemes for differential operators require a priori boundedness of the higher order derivatives, therefore the conclusions above should not be surprising. High discrete frequencies propagate with slower speeds and the numerical solution is a good approximation only at low frequencies, see e.g., \cite{Zuazua-numerics, Albin2012}. 
Also, the discretization of the phantom matters. We render the SL phantom  on a higher dimensional grid first, then  resample it to the grid we work with. This way, we have a better sampling of the continuous function the SL phantom represents. This does not affect the rates of convergence much, and without that,  we still have a very low spectral content for very small eigenvalues, even though it is not that low. The original and the reconstructions however do not show the typical aliasing artifacts like staircase edges, etc., and a still visibly ``sharp'' enough.

\subsubsection{Data not in the range} In the example above, we add the data on one of the sides to another one. The difference, and hence the new data is not in the range by unique continuation. 
The reconstruction shows obvious artifacts as one would expect. The convergence behavior as a function of $\gamma$ however, does not change in a radical way, see Figure~\ref{errors_stable_notinrange}. 
 \begin{figure}[h!] 
  \centering
  \includegraphics[trim = 0mm 5mm 0mm 0mm, clip, scale=0.5]{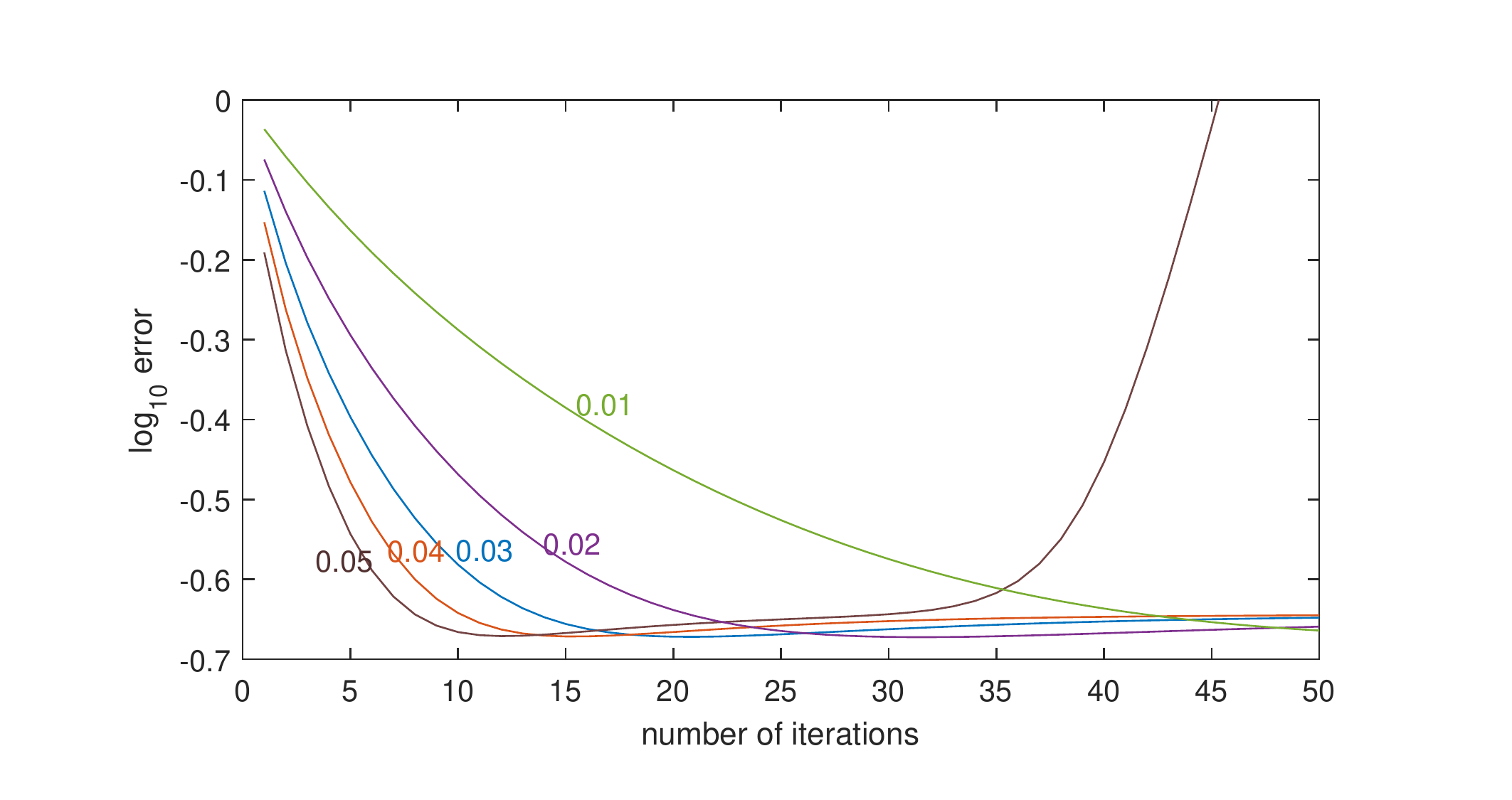}
\caption{\small Error plots in the stable case with data not in the range. Approximate convergence.   }
\label{errors_stable_notinrange}
\end{figure}  
In particular, the iterations start to diverge for $\gamma$ between $0.04$ and $0.05$.  
As we saw in Theorem~\ref{thm_noisy}, the iterations would converge to $f+\tilde f$, where $f$ is the original $f$, and $\tilde f$ is the unique solution of $\L f_1=\widetilde{\delta m}$, where $\delta m$ is the ``noise'', and $\widetilde{\delta m}$ is its orthogonal projection to $\Ran\L$. 
The error curves are consistent with the distance from $ f$ to the iterations $f_n$ converging exponentially fast to $f+\tilde f$. They have minima corresponding to the closest element of $f_n$ to $ f$.

\subsubsection{The stable case with noise} We add Gaussian noise  with  standard deviation $0.1$. The data $Lf$ ranges in the interval $[-1.2,1.7]$. The noise curves are shown in Figure~\ref{errors_stable_noise}. In the continuous model, see Theorem~\ref{thm_noisy}, the iterations would converge at a rate depending on the spectral measure of the noise; which could be very slow. They converge to $f+\tilde f$ as above.  In the numerical case however, the problem is actually unstable, as explained above. The data now has a small but not negligible spectral part in the lower part of the spectrum due to the nature of the noise. This is responsible for a very slow divergence. Note that when the ``noise'' has a different spectral character, as in Figure~\ref{errors_stable_notinrange}, we are closer to convergence.

 \begin{figure}[h] 
  \centering
  \includegraphics[trim = 0mm 25mm 0mm 10mm, clip, scale=0.5]{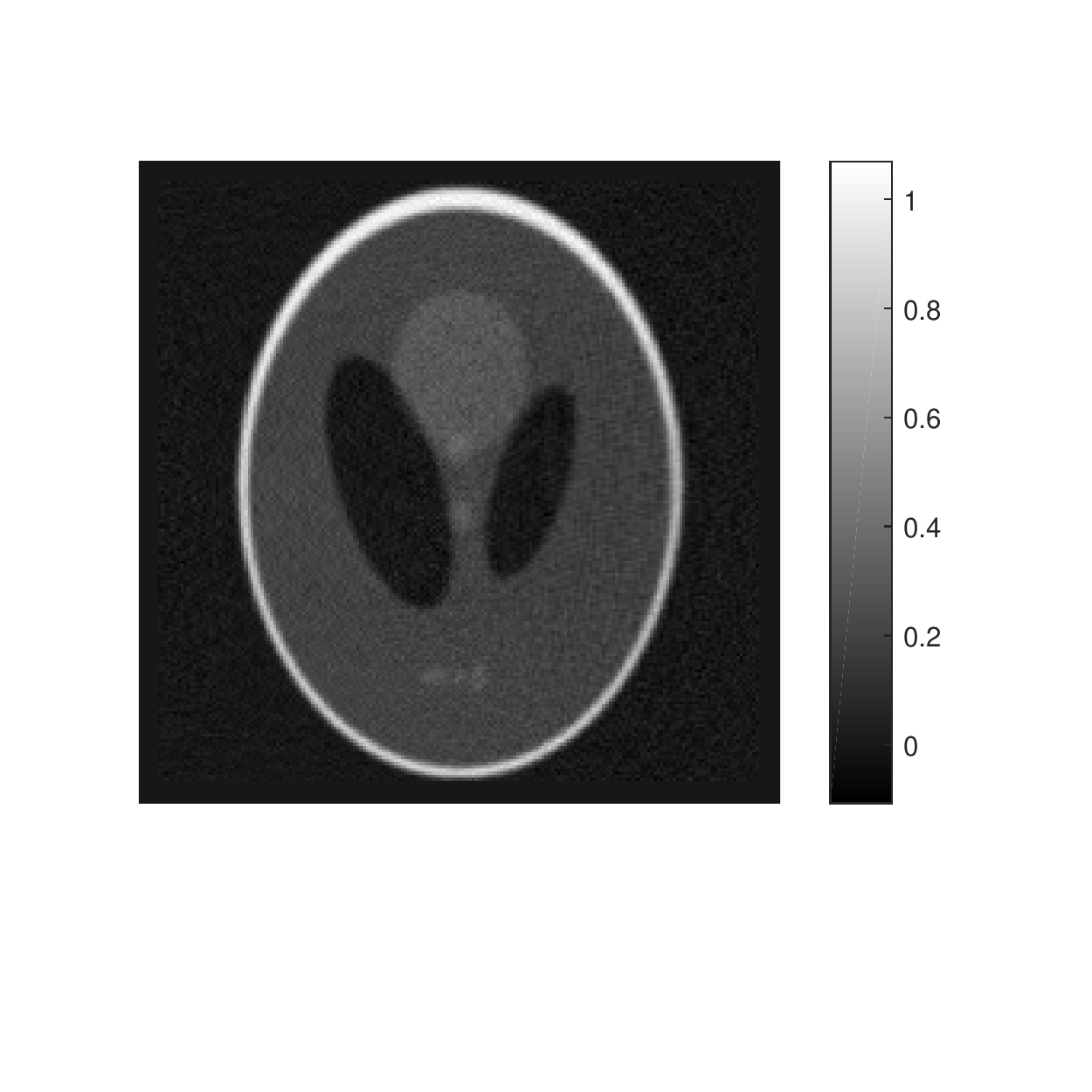}
  \includegraphics[trim = 0mm 10mm 0mm 10mm, clip, scale=0.42]{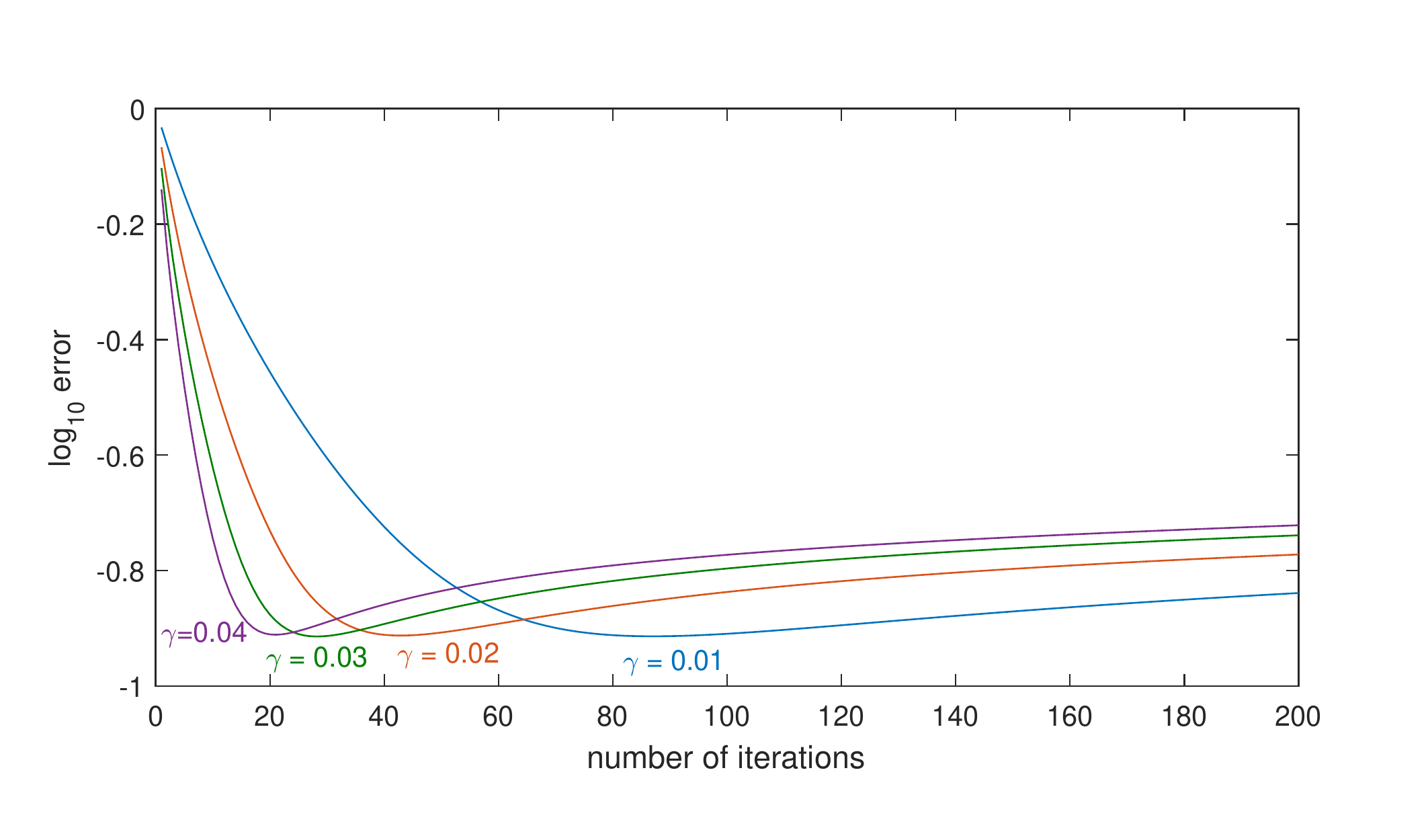}
\caption{\small Reconstruction and error plots in the stable case with noise with $200$ iterations; $\gamma=0.01,0.02,0.03,0.04$.  A slow divergence.  For $\gamma=0.05$ we get a fast divergence.}
\label{errors_stable_noise}
\end{figure}

The ATR method with this example is more sensitive to noise and gives significantly worse errors and images.


\subsubsection{Conclusions in the stable case} 
We summarize the conclusions in this section in the following. We want to emphasize again that some of them, like the first one are well known. 
\begin{itemize}
	\item The finite difference solver for the wave equation is a poor approximation of the continuous model for (discrete) frequencies close to Nyquist. We refer to  section~\ref{sec_grid}  
for a further discussion. 
	 Even the sampling of the phantom could be a poor approximation (aliased). 
	\item The discrete realization $L$ of $\L$ is unstable even if $\L$ is stable. This is not a problem with ``generic'' phantoms because they have a very small spectral content at the bottom of the spectrum of $L^*L$ which corresponds to  asmall hight frequency content (w.r.t.\ the discrete Fourier transform). 
	\item The discrete realization $L^*_{\rm w}$ of $\L^*$ used in the inversion is based on back-projection at each step and it is not the same as $L^*$. The differences do not affect the inversions visibly but create small but measurable imaginary parts of the eigenvalues of $L^*_{\rm w}L$. Most of the differences are related to modes with frequencies close to Nyquist. 
	\item Despite this, the inversion with ``generic'' phantoms works as predicted by the theory of the continuous model because they have small projections to the eigenfunctions of $L^*L$ with small eigenvalues (the unstable subspace).  If we base the theory on the discrete model, we would have to conclude that the problem is very unstable. 
	\item In case of noise, the iterations converge  (to something different than $f$) if the noise does not have low spectral content (corresponding to higher frequencies in the discrete Fourier transform). When the noise has a significant low  spectral content (high frequencies), the inherent instability of the discretization causes a slow divergence. 
\end{itemize}

\subsection{A unstable example; partial boundary data}
We take a constant speed $c=1$ so that we can compute easily the invisible singularities. We use data on a part of the boundary as indicated in Figure~\ref{3images}: the bottom and the left-hand side plus 20\% of the adjacent sides.  We take $T=1.8$ which is smaller than the side of the box $\Omega$, equal to $2$. We have uniqueness since $T_0=1.8$ in this case. For stability however, we need $T>2$. Singularities from some neighborhood of the top of the SL phantom traveling vertically and close to it, do not reach $\Gamma$ for time $T=1.8$. Therefore, such edges cannot be reconstructed stably. Our goal is to test the convergence of both methods and the dependence on $\gamma$ in the Landweber one.

 \begin{figure}[h!] 
  \centering
  \includegraphics[trim = 0mm 25mm 0mm 10mm, clip, scale=0.42
  ]{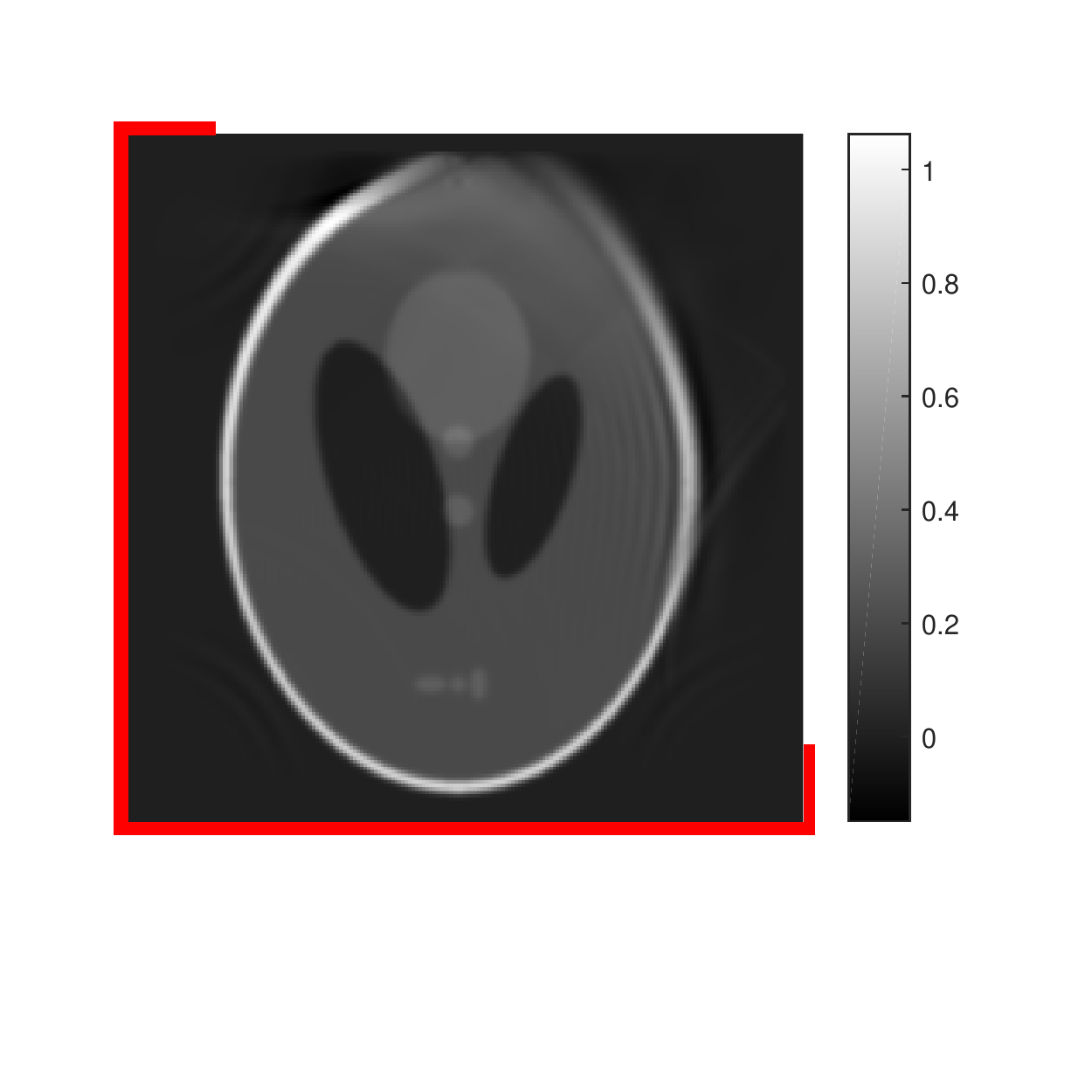}
    \includegraphics[trim = 0mm 25mm 0mm 10mm, clip, scale=0.42
  ]{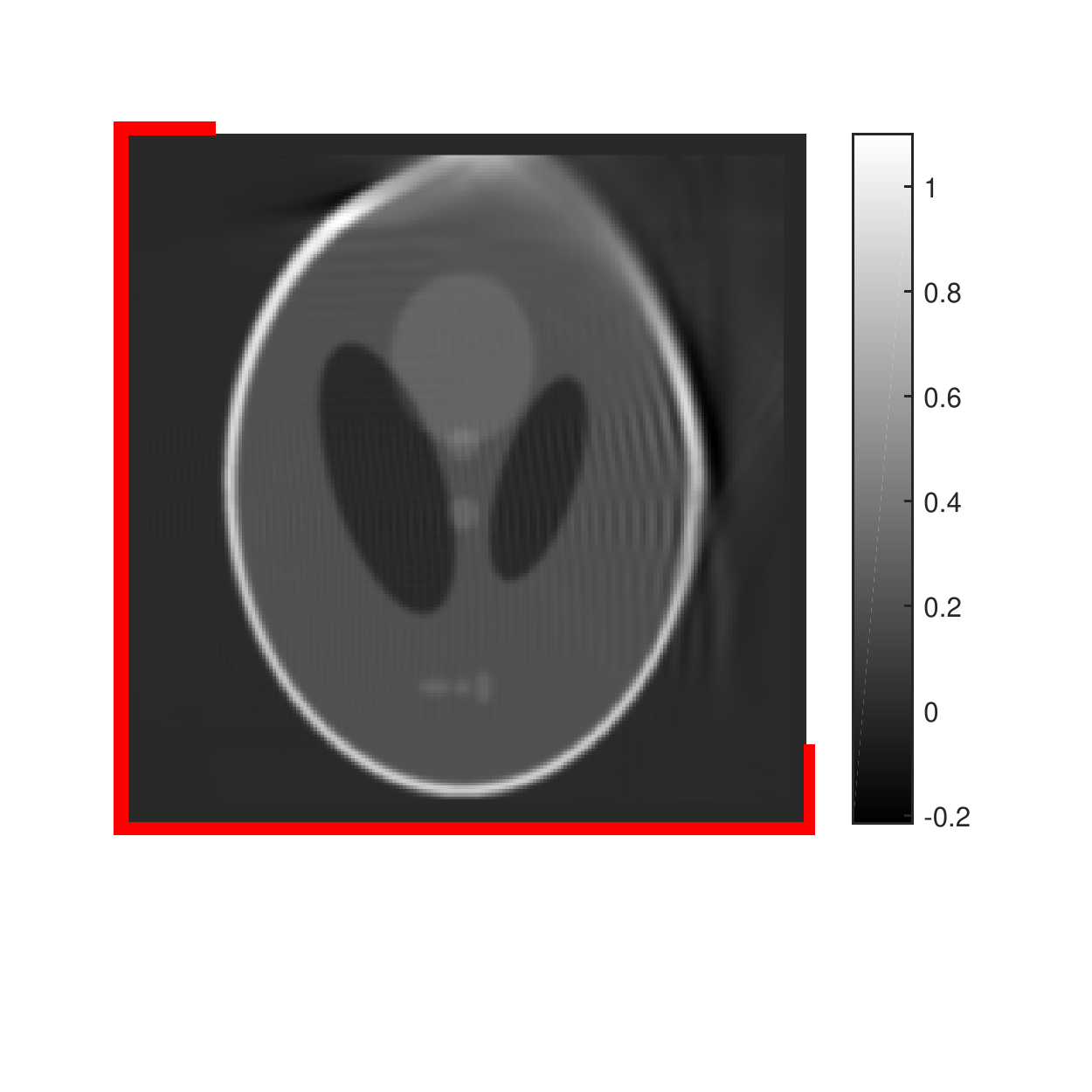}
      \includegraphics[trim = 0mm 25mm 0mm 10mm, clip, scale=0.42
  ]{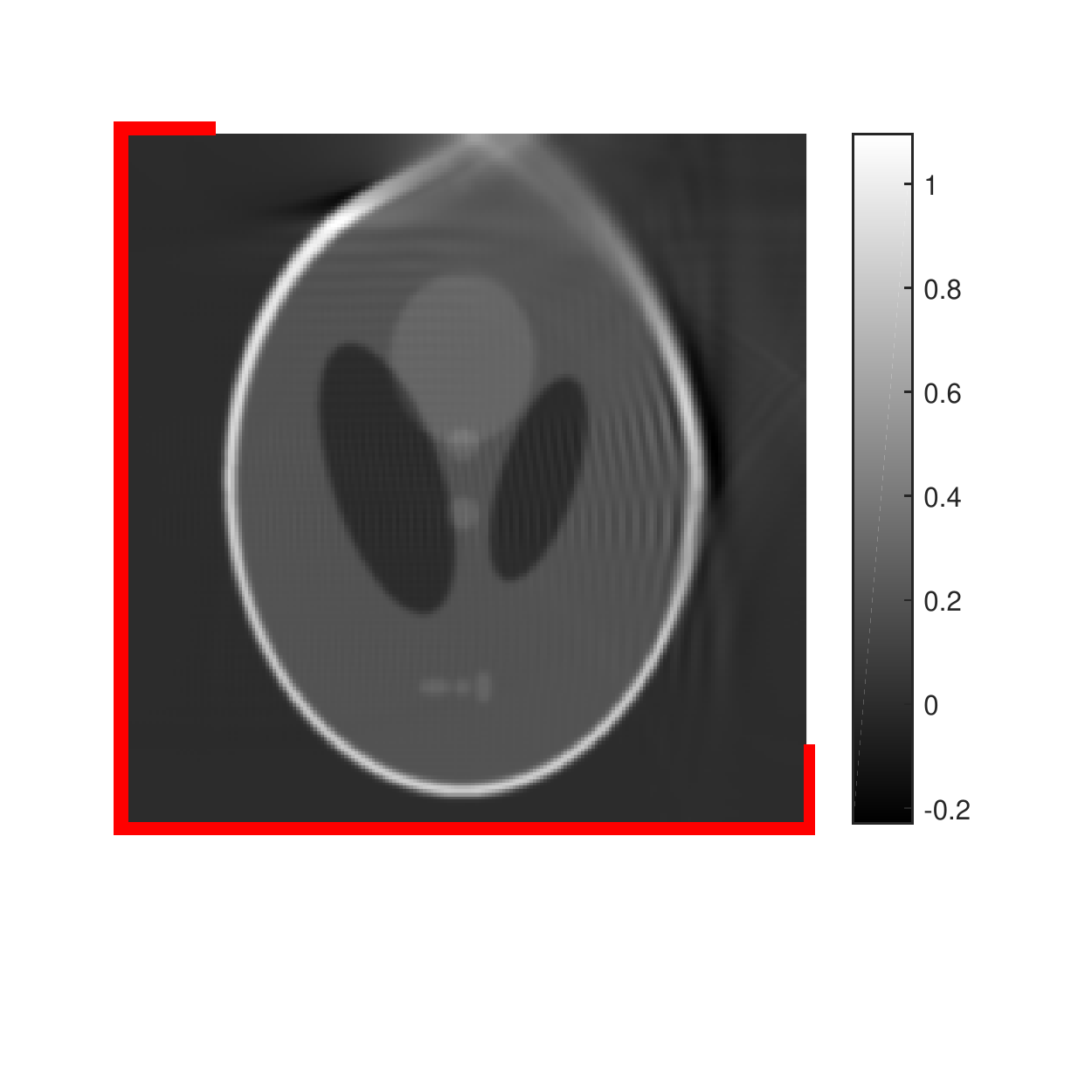}
\caption{\small Speed $1$ with $T=1.8$ with invisible singularities on the top. From left to right: (a) ATR  with 50 steps, a cut near the border, 34\% error. (b) Landweber with a cut near the borders, error $31\%$; (c) Landweber without a cut near the borders, error $36\%$.
 }
\label{3images}
\end{figure}
The reconstructions are shown in Figure~\ref{3images}. The second and the third one are done with the Landweber method but in the second one, we use the fact that we know that the SL phantom is supported in  $\Omega_0$, being a slightly smaller square. Numerically, this means that we work with $\chi L^*L\chi$ instead of $L^*L$, (i.e., $L$ is replaced by $L\chi$, where $\chi$ is the characteristic function of $\Omega_0$. In the third reconstruction, there is no such restriction. The cutoff improves the error a bit. The ``ripples'' artifacts are due to the sharp cutoff of the data at $T=1.8$. If we introduce a gradual cutoff $\chi_1(t)$ with $\chi_1(T)=0$, i.e., if we use 
$\chi L^*\chi_1 L\chi$ instead of $L^*L$, this removes the visible ``ripples'' but blurs the edge of the SL phantom in a larger neighborhood of the invisible singularities on the top. Despite the slightly larger error, the ATR reconstruction has less artifacts.

\begin{figure}[h!] 
  \centering
  \includegraphics[trim = 8mm 32mm 0mm 30mm, clip, scale=0.85]{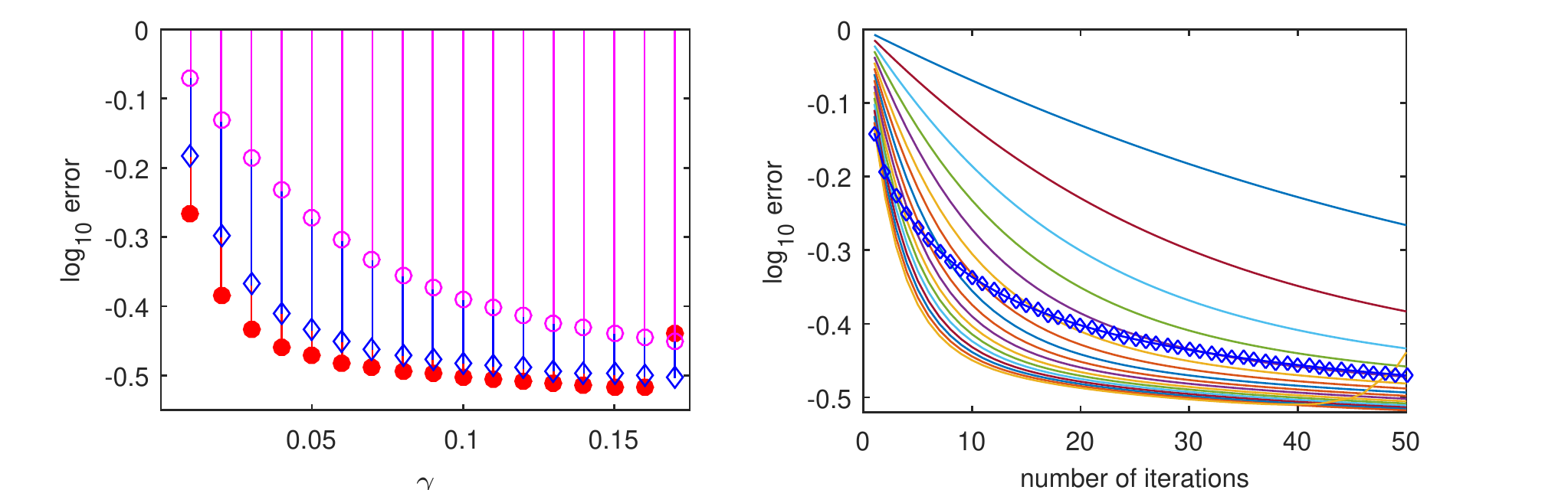}
\caption{\small Error vs.\ $\gamma$ in the unstable case.  Left: errors after $10$ iterations (boxes), after $30$ iterations (diamonds) and after $50$ ones (dots). 
The vertical axis is on a $\log_{10}$ scale and the horizontal axis represents $\gamma$. The lower curve, corresponding to 50 iterations, is flatter than in the stable case. On the right, the curve with the diamond marks is the error with the ATR method.}
\label{errors_unstable}
\end{figure}
The errors for various values of $\gamma$ are shown in Figure~\ref{errors_unstable}. The convergence is slower than in the stable case after $10$ or so steps, and the improvement with $\gamma$ increasing below reaching its optimal value is much slower. The errors are based on a cutoff to $\Omega_0$, which is the better case.  With the cutoff, the Landweber and the ATR errors are closer. On the other hand, the Landweber method is more flexible w.r.t.\ introducing such weights. The Landweber method performs better in this case in terms of the $L^2$ error.

We analyzed the eigenvalues of $L^*L$ next, and the Fourier coefficients of the SL phantom w.r.t.\ its eigenvalues. The results are shown in Figure~\ref{perc0_speed1_T1p8_Lasamatrix_spectral}. This situation is reversed now, compared to the stable case. There are still a lot of  ``zero'' (extremely small) eigenvalues with highly oscillatory eigenfunctions. The non-``zero''  Fourier coefficients however start much earlier than the non-``zero'' eigenvalues. Since our spectral analysis shows that we need to restrict our considerations on the support of the spectral measure only, we can take $\lambda\approx 800$ as a rough lower bound for this support. Then we have practically zero eigenvalues up to, roughly speaking, $\lambda_{1,800}$, where the SL phantom has a non-negligible spectral measure. The norm of the projection of the SL phantom to the space spanned by the first $1,000$ eigenvalues is around $40\%$ of the total one. This means instability and it is in contrast to Figure~\ref{fig_spectral_decomposition}.

 \begin{figure}[h!] 
  \centering
  \includegraphics[trim = 0mm 10mm 0mm 30mm, clip, scale=0.6
  ]{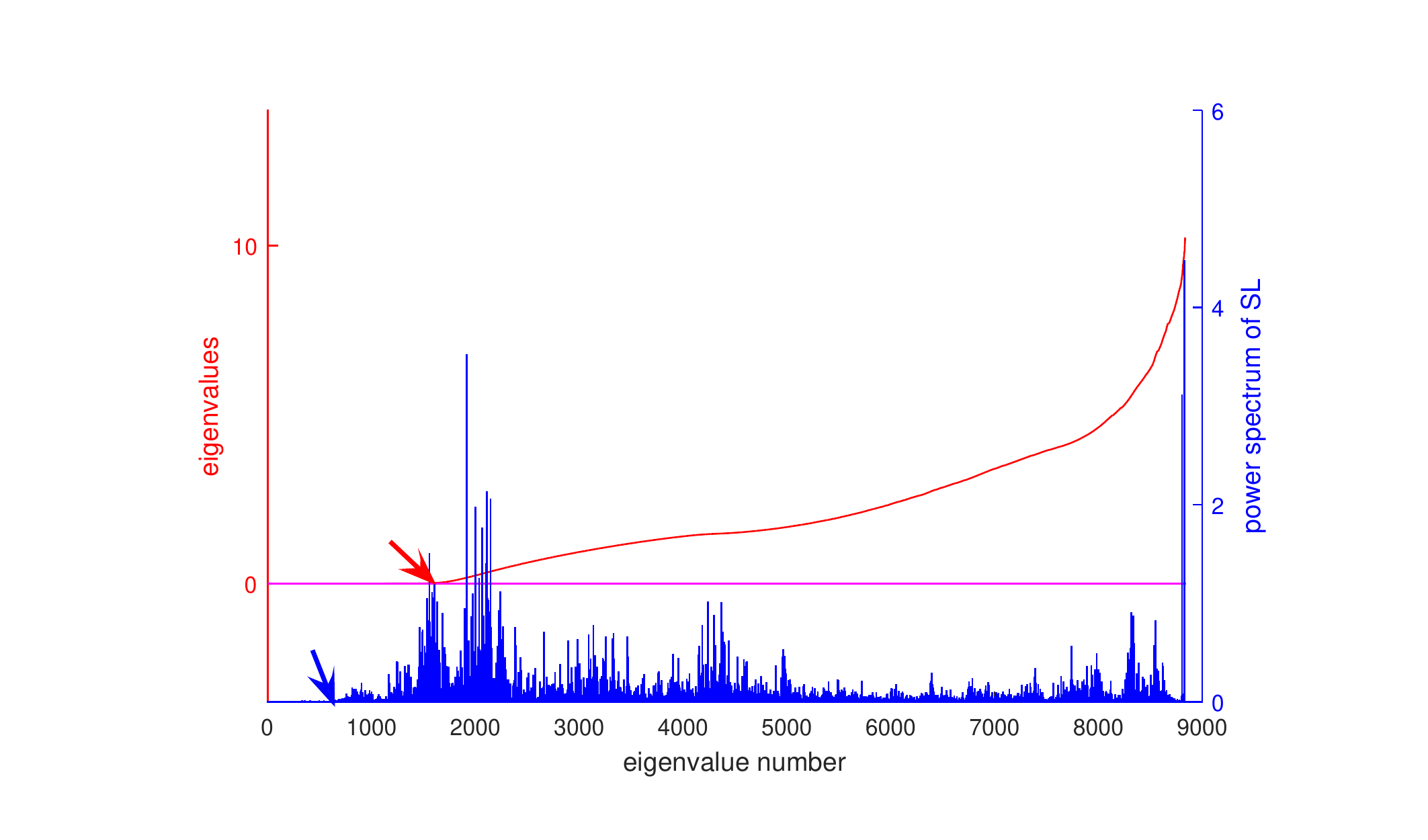}
\caption{\small The unstable case.  
The smooth increasing curve represents the eigenvalues of $\L^*\L$ in the discrete realization on an $101\!\times\!101$ grid restricted to a $94\!\times\!94$ grid. 
$\L^*$ is computed as the adjoint matrix rather than with the wave equation solver. The rough curve represents the squares of the Fourier coefficients of the SL phantom. The arrows indicate the effective lower bound of the power spectrum of the SL phantom, and the first positive eigenvalue modulo small errors, respectively. }
\label{perc0_speed1_T1p8_Lasamatrix_spectral}
\end{figure}

\subsubsection{Unstable example  with data not in the range and with noise} We add data on one side to another one, as in the stable case. The errors did not look much different than the noise case below. Next, we add Gaussian noise  with $0.1$ standard deviation. The data $Lf$ ranges in the interval $[-1,1.5]$. The errors reach a minimum and start diverging slowly for $\gamma\le0.16$, and diverge fast for $\gamma$ larger than that, see Figure~\ref{errors_unstable_noise} on the left. This suggests that $\gamma^*\approx 0.16$, for that particular image, at least, and correlates well with Theorem~\ref{thm_noisy} and its proof which proves divergence with generic perturbations of the data not in the range in the unstable case but indicates a slow divergence. Note that the optimal $\gamma$ looks similar to that in the noise free case, see Figure~\ref{errors_unstable}. On the other hand, instead of a convergent series, we get a slowly divergent one. 

Finally, we add a filter $\chi$ between $L$ and $L_{\rm w}^*$ which cuts high frequencies and also frequencies outside the characteristic cone on each side; in other words, we replace $L_{\rm w}^*$ by $L_{\rm w}^*\chi$. Since $\chi$ is a Fourier multiplier by a non-negative function, $\chi$ is a non-negative operator. The errors get smaller and the iterations do not start to diverge even after $200$ iterations, see Figure~\ref{errors_unstable_noise}. 
\begin{figure}[h!] 
  \centering
  \includegraphics[trim = -5mm 20mm 0mm 30mm, clip, scale=0.6
  ]{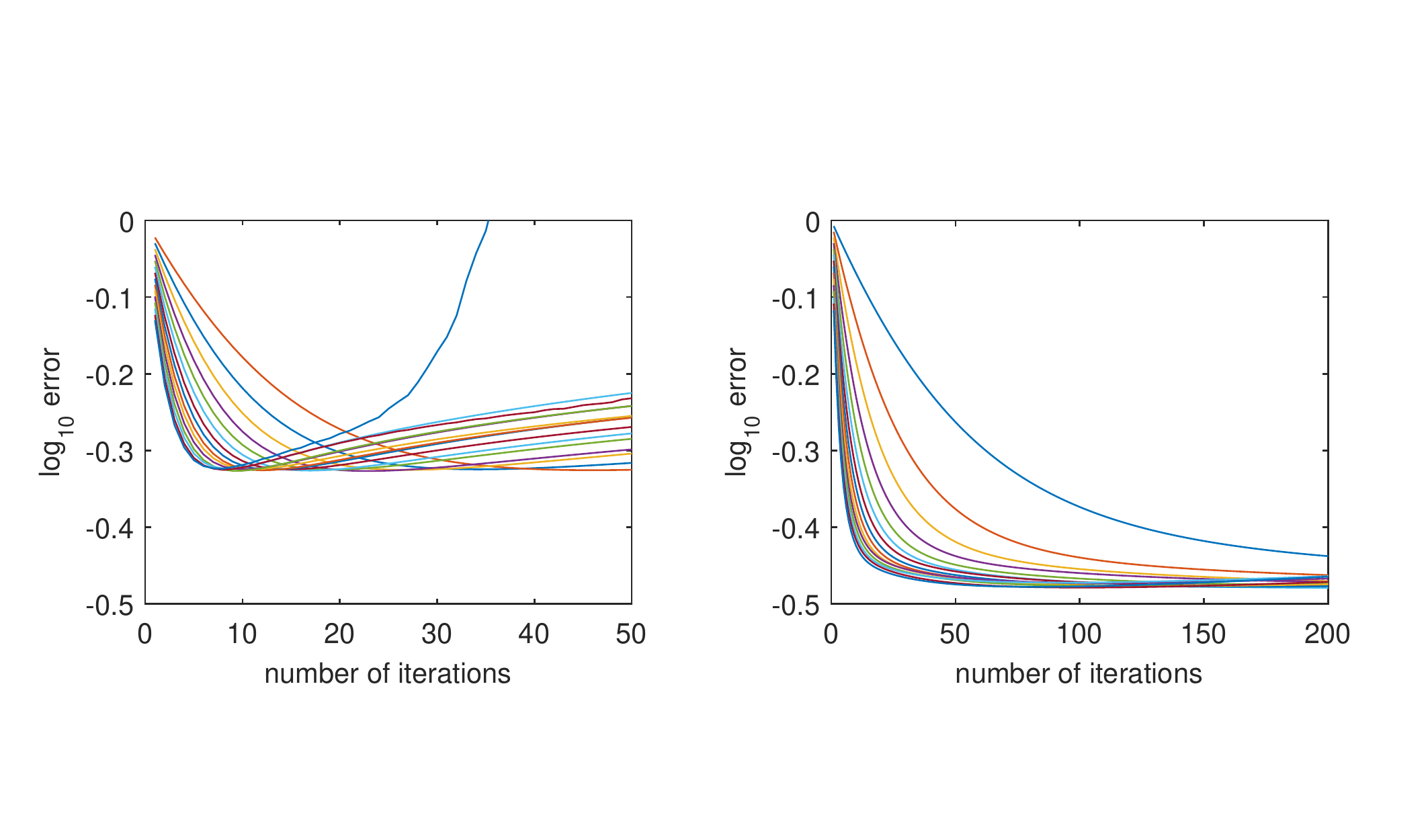}
\caption{\small A unstable case with Gaussian noise.  Left: error curves with $\gamma$ ranging from $0.03$ to $0.17$ (the diverging curve). The critical value of $\gamma$ looks close to that in the zero noise case in Figure~\ref{errors_unstable}. The iterations for $\gamma\le 0.16$   diverge slowly in contrast with the noise free case in Figure~\ref{errors_unstable}. Right: Filtered data, $\gamma\le 0.15$. The errors look like in Figure~\ref{errors_unstable}. }
\label{errors_unstable_noise}
\end{figure}
This is in line with Theorem~\ref{thm_noisy} and its proof since the filtering of the noise in the data reduced significantly the spectral content of the noisy data (w.r.t. $\L^*\L$) in the low part of the spectrum, where the instability is manifested. After that, the problem behaves as that with data not in the range but with small high frequency  content (w.r.t.\ the Fourier transform), see Figure~\ref{errors_unstable}, where the number of iterations is $50$ vs.\ $200$ in Figure~\ref{errors_unstable_noise} on the right. 

The reconstructions are shown in Figure~\ref{unstable_noise_reconstruction}. Even though the filtered reconstruction has a smaller error, it is only marginally better in recovering detail. 

\begin{figure}[h!] 
  \centering
  \includegraphics[trim = 0mm 15mm -10mm 0mm, clip, scale=0.4
  ]{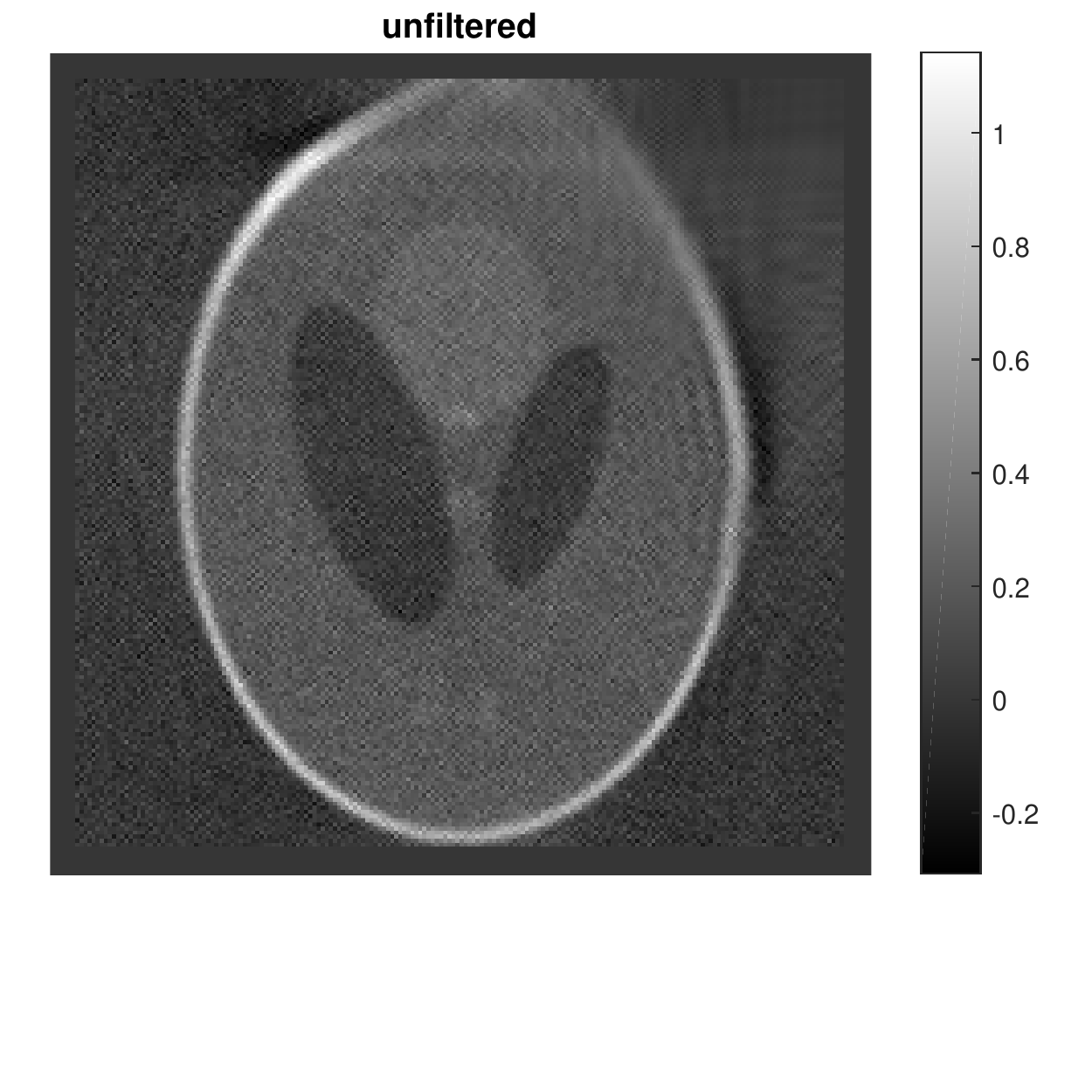}
    \includegraphics[trim =-10mm 15mm 0mm 0mm, clip, scale=0.4
  ]{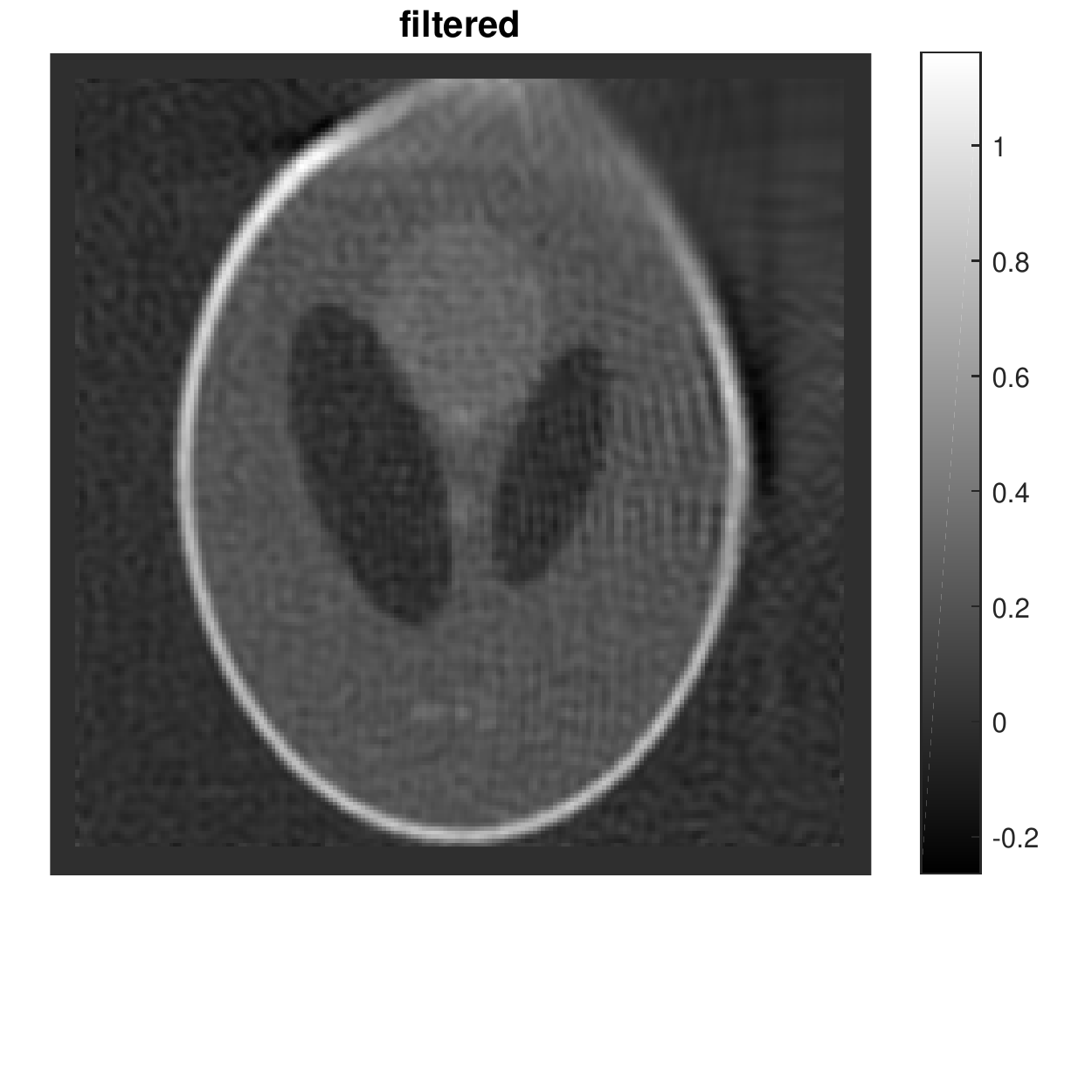}
\caption{\small The unstable case, reconstructions with the Landweber method with noise.  Left: unfiltered data. Right: filtered data. } 
\label{unstable_noise_reconstruction}
\end{figure}

\subsection{Discontinuous sound speed} We choose a sound speed with a jump across a smaller square. Such speeds model thermoacoustic tomography in brain imaging, see, e.g., \cite{YangWang2008,SU-thermo_brain,XuWang2006}.  Singularities in this case reflect from the internal boundary and may refract, as well. Since the speed outside that boundary is faster, this creates rays that do not refract. This makes the problem inside that rectangle potentially unstable. The reconstructions are shown in Figure~\ref{fig_skull}. The ATR method works a bit better not only by providing a slightly better error but there are less visible artifacts. The Landweber reconstruction handles noise better however. 

 \begin{figure}[h!] 
  \centering
        \includegraphics[trim = 10mm 20mm 10mm 20mm, clip, scale=0.5
  ]{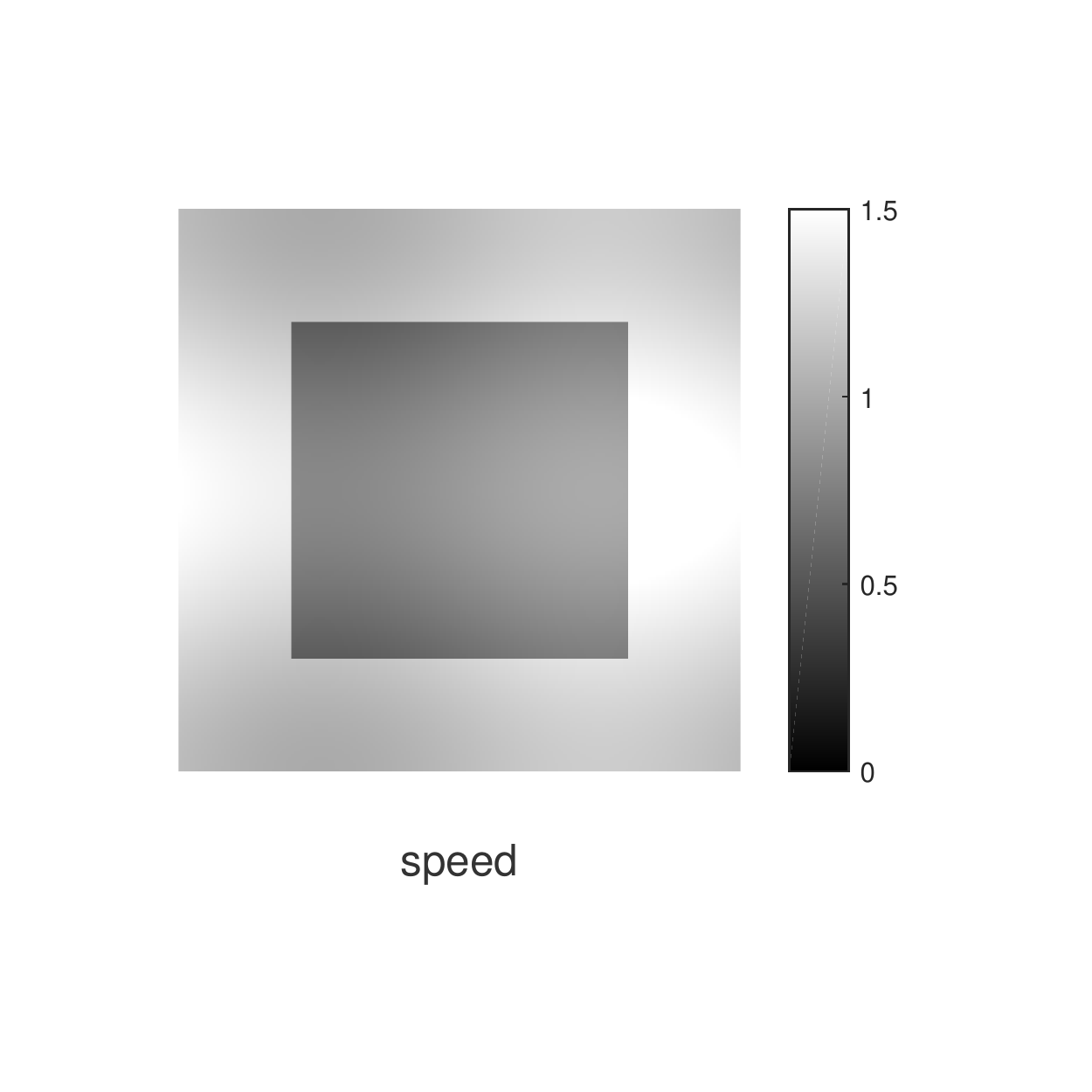}
        \includegraphics[trim = 10mm 20mm 10mm 20mm, clip, scale=0.5
  ]{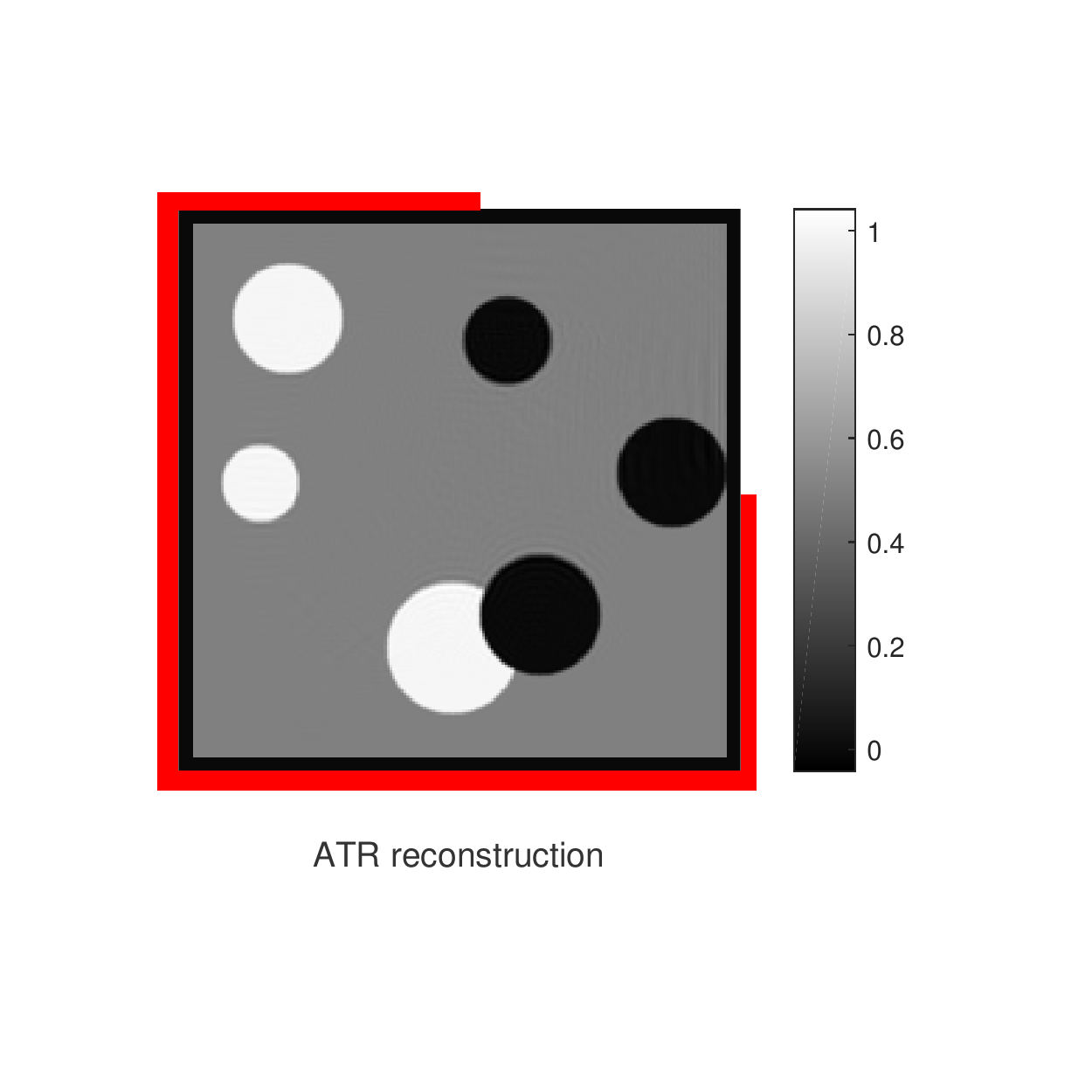}
        \includegraphics[trim = 10mm 20mm 10mm 20mm, clip, scale=0.5
  ]{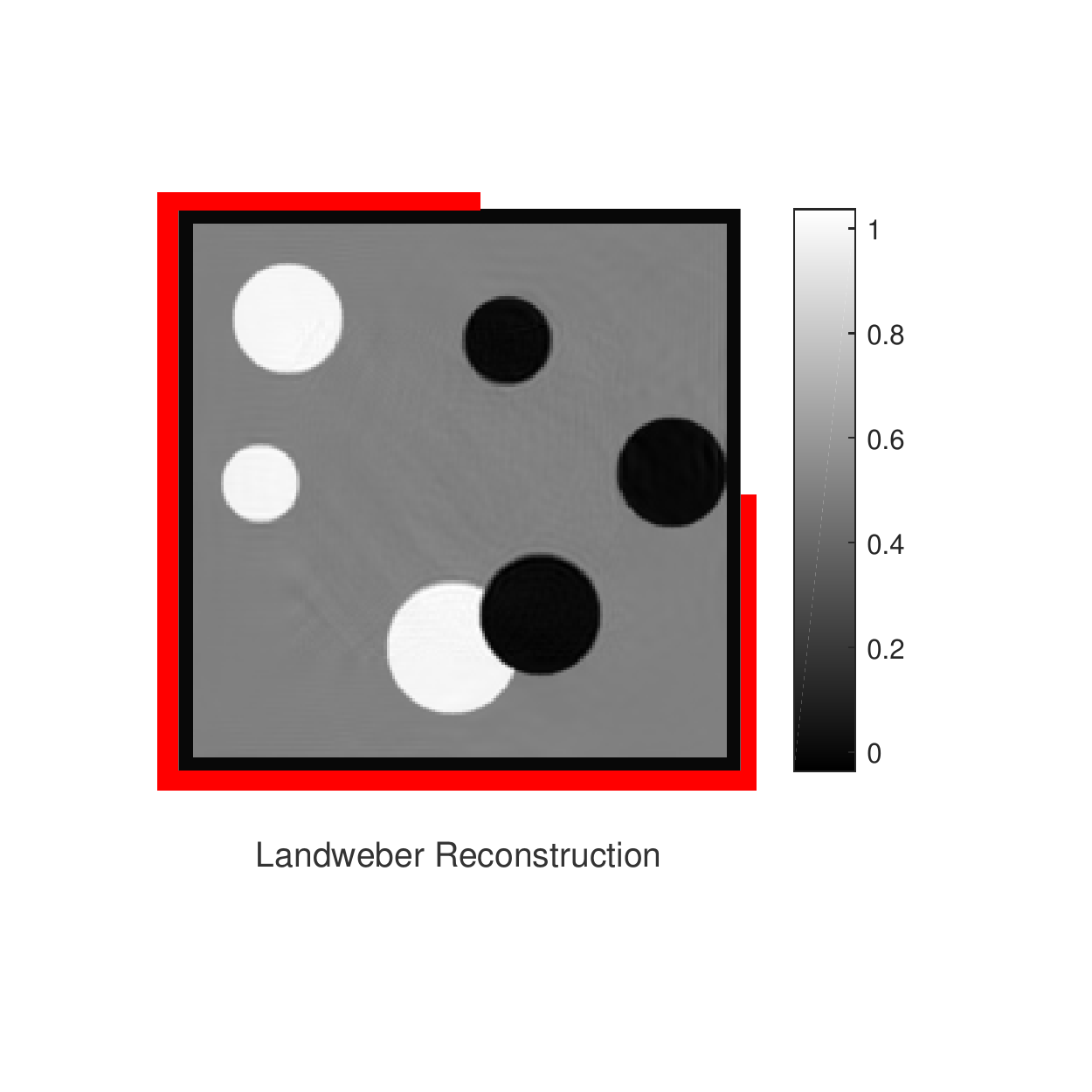}
\caption{\small Reconstructions with a discontinuous speed, plotted on the left. The original is black and white disks on a gray background. 
Data on the marked part of the boundary, $T=4$. Center: ATR reconstruction, error $1\%$. Right: Landweber reconstruction, error $1.5\%$. 
 }
\label{fig_skull}
\end{figure}

\subsection{Forward data generated on a different grid} \label{sec_grid} %
We present now examples where the data was created on a finer grid. As we mentioned above, the finite difference wave solver is inaccurate for large frequencies. This is  known as numerical (grid) dispersion, see, e.g., \cite{Alford1974, Dablain1986,Kelly1976,Fei1995,Albin2012}. High frequency waves propagate slower, which explains the terms dispersion. This effect  gets worse with time. This creates ``ripples'' in the wave fronts of high frequency waves, with oscillations in the back of the front. It was experimentally found in \cite{Alford1974, Dablain1986,Kelly1976} that in order to get a satisfactory performance for a second order finite difference scheme, say at distance 30-40 wavelengths (\cite{Dablain1986}), one needs to sample in the spatial variables each wavelength at a rate 4-5 higher than the Nyquist one, which is two point per wavelength. 
 We refer to \cite{Fei1995,Albin2012} for further discussion and for ways to improve the performance with other solvers. 

Such errors would be canceled in the backprojection, which is a well known phenomenon in numerical inverse problems.  For this reason, it is often suggested that the forward data should be generated by a different solver. This does not solve the dispersion problem however. The backward solver (if based on finite differences) would still be inaccurate for high frequencies, if they are present in the data. The different forwards solver may have the same problem. A natural attempt seems to be to generate data by a known analytic solution or using a much finer grid. This would guarantee the good accuracy of the data but the problem with the backward solver at high frequencies remains. The practical solution (again, if we backproject by finite differences) is to make sure that the data has mostly low frequency content relative to the grid used in the back-projection or to use a higher order scheme which still suffers form dispersion but allows higher frequencies \cite{Dablain1986}. Another solution would be to use other solvers not so much limited at higher frequencies but this is behind the scope of this work.

 A Landweber reconstruction of the SL phantom with data computed on a finer grid with $c=1$ is shown in Figure~\ref{SL_5p7_1p3_T2}. The ripple artifacts are due to high frequency waves in the backprojection propagating slower than the wave speed $1$. 
 Since they become worse when $T$ increases, it is important to keep $T$ low. The critical stability time here is $\sqrt{2}$ and we chose $T=2$. 

\begin{figure}[h!] 
  \centering
  \includegraphics[trim = -5mm 30mm 0mm 25mm, clip, scale=0.6
  ]{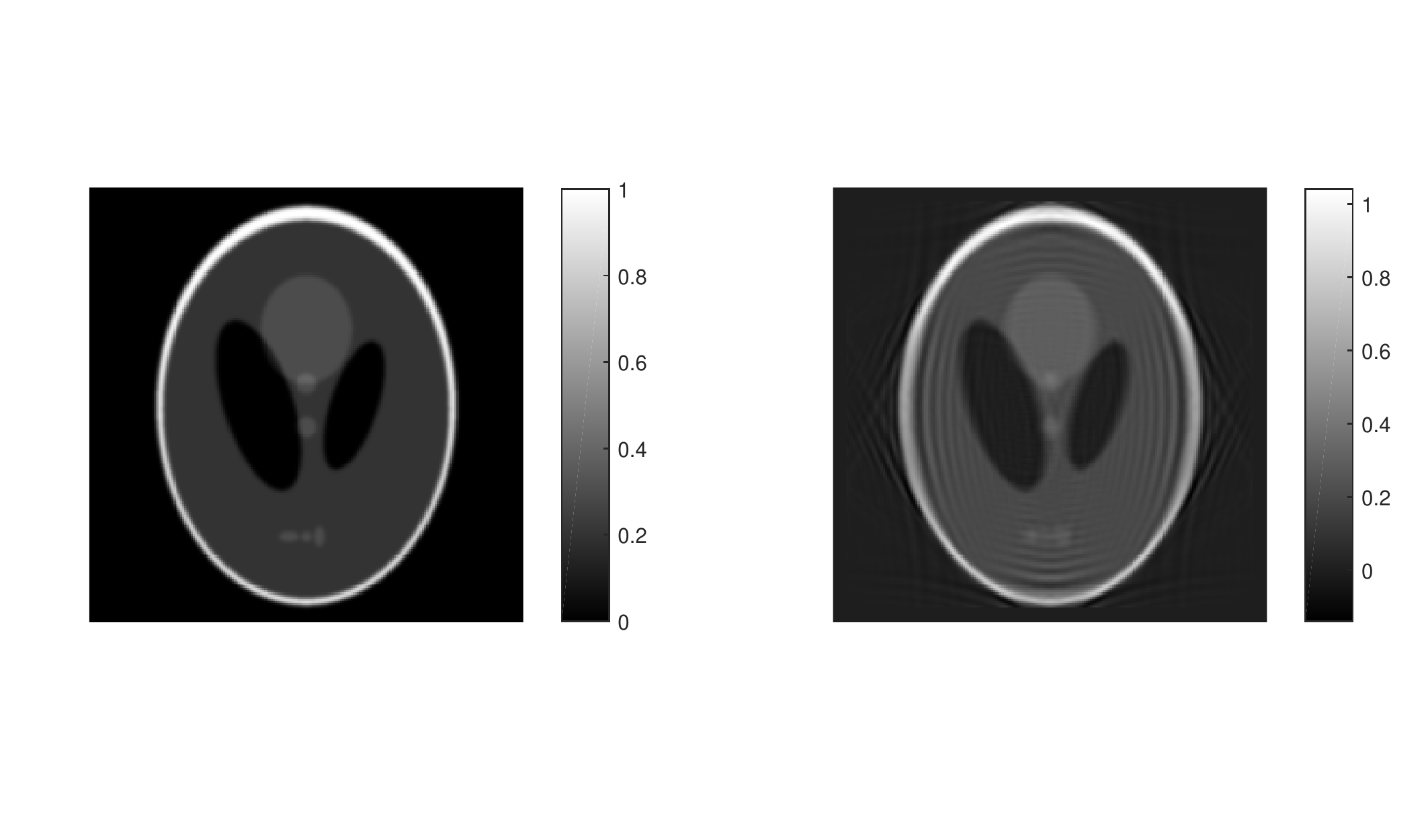}
\caption{\small Left: the original SL phantom. Right: The reconstructed  SL phantom with $c=1$ and $T=2$. The phantom was originally on a $201\times 201$ grid with $\Delta t=\Delta x /\sqrt{2}$.  
The data was computed  on a grid  rescaled by a factor of $5.7$ in the spatial variables and $7.41$ times in the time variable; and then rescaled to the the original one on the boundary before inversion. The ripple artifacts are due to high frequency waves in the time reversal moving slower than the speed $c=1$. }
\label{SL_5p7_1p3_T2}
\end{figure} 
 
  In an idealized real life situation, where we have properly sampled boundary data, textit{if we still want to use a finite difference scheme}, 
  we should do the numerical inversion on a grid much finer than the one determined by Nyquist rate of the signal. That would increase the cost, of course. To simulate such a situation, we should take a continuous phantom with a finite frequency band, oversample it by a factor of five, for example, keep $T$ low, solve the forward problem and sample it on a coarser grid depending on the frequency content of the data. To do the backprojection, we should increase the size of the grid and do finite differences. But making the grid coarser after we have computed $\Lambda f$ and finer right after it basically means that we could stay either with the original finer one in our simulations or change it but keep the data severely oversampled. As mentioned above, a more efficient solution would be to use different solvers for back propagation.

Based on the discussion above, we chose a low frequency phantom in Figure~\ref{Gauss_5p7_1p3_T2} below with data generated by the finer grid used in Figure~\ref{SL_5p7_1p3_T2}. The reconstruction is much more accurate. A reconstruction with the variable speed used above, see section~\ref{sec_5.1.1} yields visibly similar results with a relative $L^2$ error $4\%$. 

\begin{figure}[h!] 
  \centering
  \includegraphics[trim = -5mm 30mm 0mm 25mm, clip, scale=0.6
  ]{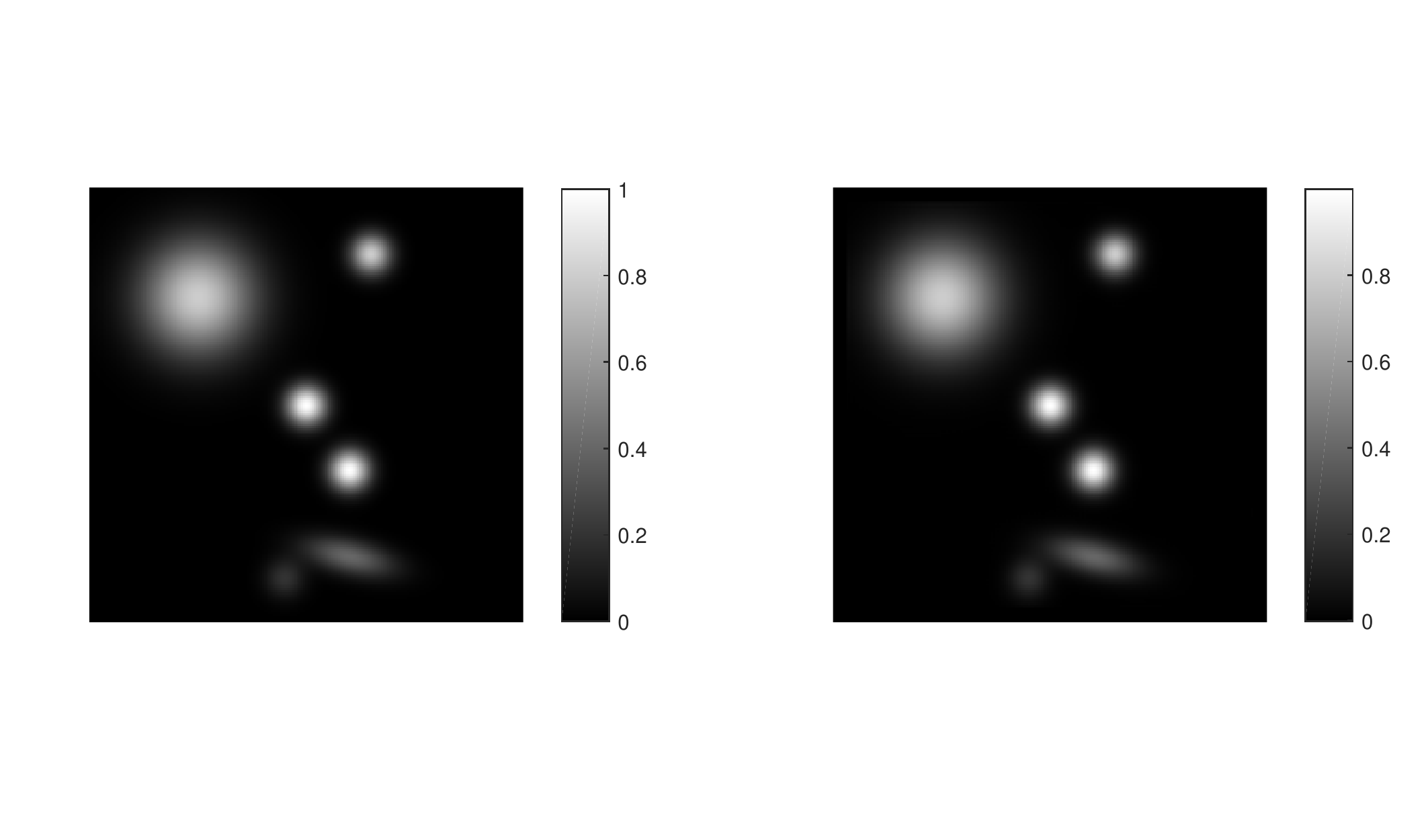}
\caption{\small Same situation as in Figure~\ref{SL_5p7_1p3_T2} but the phantom contains Gaussians with low frequency content only. Left: original. Right: reconstruction with 50 iterations. The relative $L^2$ error is $1.8\%$, the $L^\infty$ error is $3.5\%$.}
\label{Gauss_5p7_1p3_T2}
\end{figure}

\section{Review of the Averaged Time Reversal Method}\label{sec_ATR}

In this section we review briefly the averaged time reversal method proposed in \cite{St-Y-AveragedTR}. The treatment there works for any fixed Riemannian metric $g$ on $\overline{\Omega}$ and the wave equation
$$(\partial^2_t - c^2(x)\Delta_g)u=0 \quad\quad \text{ in } (0,T)\times\Omega$$ 
where $\Delta_g$ is the Laplace-Beltrami operator in the metric $g$ and $c(x)>0$ is a smooth function. In applications $g$ is the Euclidean metric and $\Delta_g$ is the Euclidean Laplacian. Here we sketch the method as well as the resulting algorithm only in the Euclidean case   for simplicity of notions. More general and detailed exposition can be found in \cite{St-Y-AveragedTR}.


\subsection{Complete data}

We start with the complete data case, i.e. when $\Gamma=\partial\Omega$. Here and below $L^2(\Omega)=L^2(\Omega,c^{-2}\d x)$.
We sometimes adopt the notation $u(t)=u(t,\cdot)$. Define the Dirichlet space $H_D(\Omega)$ as the completion of $C^\infty_0(\Omega)$ under the Dirichlet norm
$$\|f\|^2_{H_D(\Omega)}=\int_\Omega |\nabla u|^2 \,\d x.$$
Define the Neumann space $H_N(\Omega)$ to be the quotient space of $H^1(\Omega)$ modulo constant functions. For a subdomain $\Omega_0\subset\Omega$ with smooth boundary, we identify $H_D(\Omega_0)$ with the subspace of $H_D(\Omega)$ consisting of functions supported in $\overline{\Omega}_0$. Introduce the operator $\Pi_0: H_D(\Omega)\rightarrow H_D(\Omega_0)$ by
$\Pi_0 f := h$ where $h$ solves
$$\Delta h = \Delta f \quad \text{ in } \Omega_0, \quad\quad\quad h|_{\partial\Omega_0}=0.$$
$\Pi_0$ is in fact the orthogonal projection of $H_D(\Omega)$ onto $H_D(\Omega_0)$, see \cite[Lemma 2]{St-Y-AveragedTR}.

Let $u$ be the solution of \eqref{BVP} and $\mathcal{L}$ the measurement operator. We construct a (non-averaged) time reversal operator $A$ as follows. Given $h\in H^1([0,T]\times\partial\Omega)$, define $Ah:=v(0)$ where $v$ is the solution of
\begin{equation}   \label{T1}
\left\{
\begin{array}{rcll}
(\partial_t^2 - c^2(x)\Delta)v &=&0 &  \mbox{in $(0,T)\times \Omega$},\\
  v|_{(0,T)\times\bo}&=&h,\\
v|_{t=T} &=& \mathcal{P}h(T),\\ \quad \partial_t v|_{t=T}& =&0.
\end{array}
\right.               
\end{equation} 
Here $\mathcal{P}$ is the Poisson operator such that $\phi:=\mathcal{P}h(T)$ is the solution of
$$\Delta \phi=0 \quad \text{ in }\Omega, \quad\quad\quad \phi|_{\partial\Omega}=h(T).$$
The function $\phi$ is often referred to as the harmonic extension of $h(T)$.

When $h=\mathcal{L}f$ we have $v(0)=A\mathcal{L}f$, which can be viewed as an approximation of the initial data $f$ in the multiwave tomography model in $\mathbb{R}^n$ \cite{SU-thermo}. Indeed, introduce the error operator 
\begin{equation} \label{K}
Kf:=f-A\mathcal{L}f.
\end{equation}
In the multiwave tomography model in $\mathbb{R}^n$, it is shown in \cite{SU-thermo} that if $T$ is such that there is stability, $K$ is a contraction on $H_D(\Omega)$, that is, $\|K\|_{H_D(\Omega)\rightarrow H_D(\Omega)}<1$.  Moreover, $K$ is compact. This makes the operator $I-K$ invertible on $H_D(\Omega)$ and one thus deduces from \eqref{K} that $f=(I-K)^{-1}A\mathcal{L}f$. Inserting the expansion $(I-K)^{-1}=I+K+K^2+\dots$ gives the Neumann series reconstruction algorithm in \cite{SU-thermo}.

In the multiwave tomography model in $\Omega$, the error operator $K:H_D(\Omega)\rightarrow H_N(\Omega)$ is no longer a contraction regardless of how large $T$ is. In fact one could construct high frequency solutions propagating along a single broken geodesic to show that $\|K\|_{H_D(\Omega)\rightarrow H_N(\Omega)}=1$. This breaks down the inversion of $I-K$. The idea suggested in \cite{St-Y-AveragedTR} is to average the time reversal operator $A$ with respect to $T$. The averaging process causes some partial cancellation  of the microlocal singularities with opposite signs at $t=0$ when $T>T_1$, see Definition~\ref{def1}, making the averaged error operator a microlocal contraction (but not compact anymore).

We rewrite the time reversal operator $A$ for the  purpose of averaging later. Set $\tilde{v}=v-\mathcal{P}h(T)$ with $v$ the solution of \eqref{T1}, then $\tilde{v}$ solves
\begin{equation}   \label{T2}
\left\{
\begin{array}{rcll}
(\partial_t^2 -c^2(x)\Delta)\tilde v &=&0 &  \mbox{in $(0,T)\times \Omega$},\\
 \tilde v|_{(0,T)\times\bo}&=&h(t) - h(T)    ,\\
\tilde v(T) =\partial_t\tilde  v(T)& =&0. 
\end{array}
\right.               
\end{equation} 
and $Ah=v(0)=\tilde{v}(0)+\mathcal{P}h(T)$. We consider the terminal time as a parameter now, and call it $\tau$ with $\tau\in [0,T]$. We replace $T$ by $\tau$ in \eqref{T2}. Denote the corresponding solution by $\tilde{v}^\tau$ and the corresponding time reversal operator by $A(\tau)$. Note that 
\begin{equation} \label{Atau}
A(\tau)h=\tilde{v}^\tau(0) + \mathcal{P}h(\tau)
\end{equation}
and note that $\tilde{v}^\tau$ solves a similar problem as \eqref{T2} but the boundary data is replaced by $H(\tau-t)(h(t)-h(\tau))$ with $H$ the Heaviside function.

Choose a weight function $\chi\in C^\infty_0(\mathbb{R})$ which is positive and has integral equal to one over $[0,T]$. We define the averaged time reversal operator $\mathcal{A}$ by multiplying \eqref{Atau} by $\chi(\tau)$ and integrating it over $[0,T]$:
$$\mathcal{A}h:=\int^T_0 A(\tau)h\chi(\tau) \,\d\tau = \int^T_0 \tilde{v}^\tau(0)\chi(\tau)\,\d\tau + \mathcal{P}\int^T_0 h(\tau)\chi(\tau)\,\d\tau.$$ 
Finally, define $\mathcal{A}_0:=\Pi_0\mathcal{A}$ where $\Pi_0$ projects the result onto $H_D(\Omega_0)$. The projection of the second term actually vanishes since it is harmonic. The following theorem gives an explicit reconstruction of $f$ when $h=\mathcal{L}f$.

\begin{theorem}  \label{thm_average} Let $(\overline{\Omega}, c^{-2}\d x^2)$ be  non-trapping with a strictly convex boundary as a Riemannian manifold and let 
$\Omega_0\Subset\Omega$. Suppose $T>T_1(\bo,\Omega)$ and denote $\mathcal{K}_0 := \Id - \mathcal{A}_0\mathcal{L}$ on $H_D(\Omega_0)$.
Then $\|\mathcal{K}_0\|_{H_D(\Omega_0)\rightarrow H_D(\Omega_0)}<1$. 
In particular, $\Id-\mathcal{K}_0$ is invertible on $H_{D}(\Omega_0)$, and the inverse  problem has an explicit solution of the form
\be{2.2}
f = \sum_{m=0}^\infty \mathcal{K}_0^m \mathcal{A}_0h, \quad h:= \mathcal{L} f.
\ee
\end{theorem}

This theorem leads to the following iterative reconstruction algorithm:
\be{NSI}
\begin{split}
f_1 &=  \mathcal{A}_0 h, \quad h:= \mathcal{L} f,\\
f_{n}& = (\Id-\mathcal{A}_0\mathcal{L}) f_{n-1} +\mathcal{A}_0 h, \quad n=2,3,\dots. 
\end{split}
\ee

\subsection{Partial data}

In the partial data case where $\Gamma\subsetneq\partial\Omega$, the (non-averaged) time reversal operator $A$ needs the following modification. Given $h\in H^1([0,T]\times\partial\Omega)$ with $\supp h \subset (0,T)\times\Gamma$, we define $Ah:=v(0)$ where $v$ is the solution of the following problem with mixed boundary conditions
\[
\left\{
\begin{array}{rcll}
(\partial_t^2 - c^2(x)\Delta)v &=&0 &  \mbox{in $(0,T)\times \Omega$},\\
  v|_{(0,T)\times\Gamma}&=&h,\\
   \partial_\nu  v|_{(0,T)\times(\bo\setminus\Gamma)}&=&0,\\
v|_{t=T} &=&\phi,\\ \quad \partial_t v|_{t=T}& =&0, 
\end{array}
\right.               
\]
Here $\nu$ is the unit outer normal vector field of $\partial\Omega$; $\phi$ is the solution of the Zaremba problem
$$\Delta \phi=0 \quad \text{ in } \Omega, \quad\quad\quad \phi|_{\Gamma}=h(T), \quad \partial_\nu v|_{\partial\Omega\backslash\Gamma}=0.$$
In other words, the modification to the time reversal \eqref{T1} is that we impose the homogeneous Neumann boundary condition (which is satisfied by the forward solution $u$) on $\partial\Omega\backslash\Gamma$ where the measurement is unavailable. The averaged time reversal operator is then constructed in a similar manner. We were only able to prove however that the averaging method with partial data however provides a parametrix recovering the singularities of $f$ which do not hit the edge of $[0,T]\times\Gamma$ \cite{St-Y-AveragedTR}. This does not cause a significant loss of singularities as those that hit the edge of $\Gamma$ form  a set of measure zero in the cotangent bundle. Then we have $\|K\|\le1$ but we do not know if $\|K\|<1$. This does not guarantee convergence of the Neumann series but we use such series numerically nevertheless. 

%

\end{document}